\documentclass[11pt]{article}
\usepackage{amsmath,amsthm,amsfonts,amssymb}
\usepackage[all]{xy}
\usepackage[usenames]{color} 
\usepackage{pdfsync}

\usepackage{xr}

\usepackage{xr}

\newcommand\brefGM[1]{\ref{GM#1}}
\newcommand\brefRK[1]{\ref{RK#1}}
\newcommand\brefWA[1]{\ref{WA#1}}

\newcommand\refGM[1]{I.\brefGM{#1}}
\newcommand\refRK[1]{II.\brefRK{#1}}
\newcommand\refWA[1]{III.\brefWA{#1}}

\newcommand\prefGM[1]{(\brefGM{#1})}

\newcommand\prefWA[1]{(\brefWA{#1})}

\newcommand\trefGM[2]{\refGM{#1}~\prefGM{#2}}

\newcommand\trefWA[2]{\refWA{#1}~\prefWA{#2}}

\newcounter{enumitemp}
\newenvironment{enumeratecontinue}{
 \setcounter{enumitemp}{\value{enumi}}
 \begin{enumerate}
 \setcounter{enumi}{\value{enumitemp}}
}
{
 \end{enumerate}
}

\newcommand\pref[1]{(\ref{#1})}

\newtheorem{thm}{Theorem}[section]

\newtheorem*{theorem*}{Theorem}

\newtheorem*{TheoremJ}{Theorem J}

\newtheorem*{propositionUnivAttrInd}{Inductive Step of Proposition \ref{PropUniversallyAttracting}}

\newtheorem*{propositionDrivingUpInd}{Inductive Step of Proposition \ref{PropDrivingUp}}

\newtheorem{lemma}[thm]{Lemma}

\newtheorem{proposition}[thm]{Proposition}
\newtheorem*{proposition*}{Proposition}

\newtheorem{prop}[thm]{Proposition}

\theoremstyle{definition}

\newtheorem{definition}[thm]{Definition} 

\newtheorem*{defn*}{Definition}

\newtheorem{remark}[thm]{Remark}
\theoremstyle{remark}

\newcounter{remarks}
{\paragraph*{Remarks}\smallskip
 \begin{list}{\arabic{remarks}. }{\usecounter{remarks}%
 \setlength{\leftmargin}{0in}%
 \setlength{\rightmargin}{0in}%
 \setlength{\labelsep}{0pt}%
 \setlength{\labelwidth}{0pt}%
 \setlength{\listparindent}{0pt}%
 }
}
{
\end{list}
}

\newcommand\from\colon
\newcommand\inv{{-1}}
\newcommand\subgroup{<}
\newcommand\infinity\infty
\newcommand\na{\text{na}}
\newcommand\supp{\text{supp}}
\newcommand\disjunion\coprod

\DeclareMathOperator{\cl}{cl}

\DeclareMathOperator\IA{IA}
\DeclareMathOperator\Acc{Acc}

\newcommand\IAThree{\IA_n(\Z/3)}

\newcommand{\A}{\mathcal A}
\newcommand\B{\mathcal B}
\newcommand{\N}{{\mathbb N}}

\newcommand{\Z}{{\mathbb Z}}

\renewcommand\H{{\mathcal H}}
\newcommand{\h}{\H}
\newcommand\cH\h

\newcommand{\V}{\mathcal V}

\newcommand{\Out}{\mathsf{Out}}

\newcommand{\Stab}{\mathsf{Stab}}

\newcommand{\F}{\mathcal F}

\renewcommand\L{\mathcal L}

\newcommand{\fG} {f : G \to G}
\newcommand{\ti} {\tilde}

\newcommand{\eg}{EG}
\newcommand{\noneg}{NEG}
\renewcommand\neg\noneg

\newcommand{\wt}{\widetilde}

\newcommand{\ct}{CT}
\newcommand{\cts}{CTs}

\newcommand{\comment}[1]{}

\newcommand\BookOne{\cite{BFH:TitsOne}}

\newcommand\Intro{\cite{HandelMosher:SubgroupsIntro}}
\newcommand\PartOne{Part I \cite{HandelMosher:SubgroupsI}}
\newcommand\PartTwo{Part II \cite{HandelMosher:SubgroupsII}}
\newcommand\PartThree{Part III \cite{HandelMosher:SubgroupsIII}}

\DeclareMathOperator\interior{int}
\newcommand\bdy\partial

\newcommand\intersect\cap
\newcommand\union\cup
\newcommand\<\langle
\renewcommand\>\rangle
\newcommand\meet\wedge
\newcommand\composed{\circ}
\newcommand\cross\times
\newcommand\restrict{\bigm |}
\newcommand\wh{\widehat}
\newcommand\inject\hookrightarrow

\newcommand\Id{\text{Id}}

 \newcommand\surjection\twoheadrightarrow
 
\newcommand\suchthat{\bigm|}
\newcommand\hyp{\mathbf{H}}
\DeclareMathOperator\MCG{\mathcal{MCG}}

\newcommand{\Lambdapmp}{\Lambda^{\pm}_\phi}

\newcommand\absolute{\text{\tiny{\,abs}}}
\newcommand\relative{\text{\tiny{\,rel}}}

\title{Subgroup decomposition in $\Out(F_n)$\\ Part IV: Relatively irreducible subgroups}

\author{Michael Handel and Lee Mosher}

\begin{document}

\maketitle

\begin{abstract}
This is the fourth and last in a series of four papers, announced in \cite{HandelMosher:SubgroupsIntro}, that develop a decomposition theory for subgroups of $\Out(F_n)$.

In this paper we develop general ping-pong techniques for the action of $\Out(F_n)$ on the space of lines of $F_n$. Using these techniques we prove the main results stated in \cite{HandelMosher:SubgroupsIntro}, Theorem~C and its special case Theorem~I, the latter of which says that for any finitely generated subgroup $\h \subgroup \Out(F_n)$ that acts trivially on homology with $\Z/3$ coefficients, and for any free factor system $\F$ that does not come from a complementary pair of free factors of $F_n$ nor from a rank~$n-1$ free factor, if $\h$ is fully irreducible relative to $\F$ then $\h$ has an element that is fully irreducible relative to~$\F$. We also prove Theorem~J which, under the additional hypothesis that $\h$ is geometric relative to~$\F$, describes a strong relationship between $\h$ and a mapping class group of a surface. 
\end{abstract}

\subsection*{Statements of results}

Recall from the introduction \Intro\ the main theorem of this series, which we have restated with a broader hypothesis:
\begin{description}
\item[Theorem C.] 
Let $\H \subgroup \IA_n(\Z/3)$ be a subgroup and $\F \sqsubset \F'$ an $\h$-invariant multi-edge extension of free factor systems such that $\H$ is irreducible relative to $\F \sqsubset \F'$. If $\H$ is finitely generated, or if there exists $\theta \in \H$ and a lamination pair $\Lambda^\pm_\theta \in \L^\pm(\theta)$ that is carried by $\F'$ but not by $\F$, then there exists $\phi \in \h$ which is fully irreducible relative to $\F \sqsubset \F'$. 
\end{description}
We recall also from \Intro\ a separately stated special case, which we here strengthen by broadening the hypothesis, and by adding a moreover statement for further applications. If $\F$ is an $\H$-invariant free factor system then $\H$ is \emph{geometric above} $\F$ (Definition~\ref{DefGeometricAboveF}) if for each $\phi \in \h$, each nongeometric lamination pair in $\L^\pm(\phi)$ is supported by~$\F$ (see \PartOne\ for material on geometric lamination pairs).
\begin{description}
\item[Theorem I.]
Let $\H \subgroup \IA_n(\Z/3)$ be a subgroup and  $\F \sqsubset \{[F_n]\}$   an $\h$-invariant multi-edge extension of free factor systems such that $\H$ is irreducible relative to $\F$. If $\H$ is finitely generated, or if there exists $\theta \in \H$ and a lamination pair $\Lambda^\pm_\theta \in \L^\pm(\theta)$ that is not carried by $\F$,  then there exists $\phi \in \H$ which is fully irreducible relative to $\F$. Moreover, if     either $\Lambda^-_\theta$ is non-geometric or $\H$ is geometric above $\F$ then for any weak neighborhood $U \subset \B$ of a generic leaf of $\Lambda^-_\theta$ we may choose $\phi$ so that generic leaves of $\Lambda^-_\phi$, the unique element of $\L(\phi^{-1})$ that is not carried by $\F$, are contained in $U$.
\end{description}

In this paper we prove these theorems and the closely related Theorem~J, a version of Theorem~I that applies under the additional hypothesis that $\h$ is geometric  above~$\F$. 

Recall that for $\h$ or $\phi$ to be irreducible relative to $\F \sqsubset \F'$ (when $\F'=\{[F_n]\}$ it is dropped from the notation) means that there is no $\h$-invariant free factor system strictly between $\F$ and $\F'$; and relatively fully irreducible means relatively irreducible for all finite index subgroups or finite powers, respectively, of $\h$ or $\phi$. Throughout this paper, in the context of $\IA_n(\Z/3)$ we often drop the adjective ``fully'' and simply write ``irreducible'' --- this is justified by applying Theorem~B aka Theorem~\refRK{ThmPeriodicFreeFactor},\footnote{Cross references such as ``Theorem II.X.Y'' refer to Theorem X.Y of \PartTwo. Cross references to the Introduction \Intro, to \PartOne, and to \PartTwo\ are to the June 2013 versions.} 
which says that for each $\phi \in \IA_n(\Z/3)$, a free factor system $\F$ is $\phi$-periodic if and only if it is~$\phi$-invariant, and so a subgroup or element of $\IA_n(\Z/3)$ is fully irreducible relative to an invariant extension of free factor systems $\F \sqsubset \F'$ if and only if it irreducible relative to that extension.

\subsection*{Outline and contents.} Here is an overview of Part~IV, consisting mostly of a somewhat detailed outline of the proof of Theorem~I. The reader may prefer to review the briefer introduction to Part~IV found in \Intro, and then go right to main body of the paper beginning in Section~\ref{SectionFindingAttrLams}. 

\smallskip\textbf{Section \ref{SectionPingPongArgument}. The ping-pong argument.} Consider $\h$ and $\F$ as in Theorem~I. The proof of that theorem depends on a ping-pong game described in Proposition~\ref{PropSmallerComplexity}, which is based in turn on the weak attraction theory developed in \PartThree. The ping-pong game has two players, two elements $\phi,\psi \in \h$ equipped with laminations pairs $\Lambda^\pm_\phi \in \L^\pm(\phi)$, $\Lambda^\pm_\psi \in \L^\pm(\psi)$ neither of which are carried by~$\F$. Since $\F$ is invariant under both $\phi$ and $\psi$, both nonattracting subgroup systems $\A_\na\Lambda^\pm_\phi$, $\A_\na\Lambda^\pm_\psi$ carry~$\F$. The conclusion of Proposition~\ref{PropSmallerComplexity} is the existence of large exponents $l,m > 0$ so that $\xi = \psi^l \phi^m$ has a lamination pair $\Lambda^\pm_\xi$ whose nonattracting subgroup system $\A_\na \Lambda^\pm_\xi$ also carries $\F$, and such that the following hold: $\A_\na\Lambda^\pm_\xi$ is carried by each of $\A_\na \Lambda^\pm_\phi$ and $\A_\na \Lambda^\pm_\psi$; the laminations $\Lambda^+_\xi, \Lambda^+_\psi$ are close in the weak topology; and the laminations $\Lambda^-_\xi,\Lambda^-_\phi$ are also close.

Proposition~\ref{PropSmallerComplexity} achieves these conclusions under the hypotheses that the positive laminations are weakly attracted to each other under forward iteration and the negative laminations are weakly attracted to each other under negative interation. Proposition~\ref{PropSmallerComplexity} has further hypotheses and conclusions designed to control geometric behavior, and to control the ``closeness'' of $\Lambda^+_\xi$ to $\Lambda^+_\psi$ and of $\Lambda^-_\xi$ to $\Lambda^-_\phi$. 

\smallskip\textbf{Section~\ref{SectionConstructingConjugator}. Constructing a conjugator.} A single ping-pong player needs another player in order to have a game. Starting with only a single player, an element $\phi \in \h$ with lamination pair $\Lambda^\pm_\phi \in \L^\pm(\phi)$ not carried by~$\F$, for a second player one uses a conjugate $\psi = \zeta \phi \zeta^\inv$ with lamination pair $\Lambda^\pm_\psi = \zeta(\Lambda^\pm_\phi)$. In order for the hypotheses of Proposition~\ref{PropSmallerComplexity} to be satisfied, the conjugator $\zeta \in \h$ must scramble up the given data regarding~$\phi$: $\zeta$ must not preserve $\A_\na\Lambda^\pm_\phi$, it must not map generic leaves of $\Lambda^\pm_\phi$ into the nonattracting subgroup system $\A_\na\Lambda^\pm_\phi$, and a few other useful properties of $\zeta$ must hold. Such a $\zeta$ need not exist in general, certainly not if $\h$ stabilizes $\A_\na\Lambda^\pm_\phi$, so Lemma~\ref{LemmaConjugatorConstructor} has a strong hypothesis requiring that the subgroup of $\h$ that stabilizes $\A_\na(\Lambda^\pm_\phi)$ has infinite index in~$\h$.

\smallskip\textbf{Sections~\ref{SectionProofUnivAttr} and~\ref{SectionLooking}: Driving down $\A_\na \Lambda^\pm_\phi$ and driving up $\F_\supp(\Lambda^\pm_\phi)$.} The proof of Theorem~I follows different courses depending on whether $\h$~is geometric above~$\F$. 

If $\h$ is not geometric above $\F$, the proof of Theorem~I requires two ping-pong tournaments. Each round of each tournament starts with one player $\phi \in \h$ and nongeometric lamination pair $\Lambda^\pm_\phi$ not carried by $\F$. A second player $\psi = \zeta \phi \zeta^\inv$ with pair $\Lambda^\pm_\psi = \zeta(\Lambda^\pm_\psi)$ is found using a conjugator $\zeta$ chosen according to Lemma~\ref{LemmaConjugatorConstructor}. Under these conditions, the free factor systems $\A_\na(\Lambda^\pm_\phi)$ and $\A_\na(\Lambda^\pm_\psi)$ each contain~$\F$. The first tournament, described in Proposition~\ref{PropUniversallyAttracting}, is an inductive procedure for driving down the free factor system $\A_\na(\Lambda^\pm_\phi)$: one iteratively applies the ping-pong result Proposition~\ref{PropSmallerComplexity}, producing $\xi = \psi^l \phi^m$ and a nongeometric $\Lambda^\pm_\xi$ so that $\A_\na(\Lambda^\pm_\xi)$ still contains $\F$ but is strictly contained in $\A_\na(\Lambda^\pm_\phi)$ in the sense of the partial ordering $\sqsubset$. Strictly descending chains of free factor systems under $\sqsubset$ must terminate, and by induction one obtains the conclusion of Proposition~\ref{PropUniversallyAttracting}: there exists a nongeometric $\phi \in \h$ and $\Lambda^\pm_\phi$ such that the free factor system $\A_\na\Lambda^\pm_\phi$ is minimal subject to the requirement of carrying $\F$, meaning simply that $\A_\na\Lambda^\pm_\phi = \F$. The second tournament, described in Proposition~\ref{SectionLooking}, is a similar inductive procedure to drive up the absolute free factor support $\F_\supp(\Lambda^\pm_\phi)$, subject to the requirement that $\A_\na(\Lambda^\pm_\phi)$ stays fixed at~$\F$. Once $\F_\supp(\Lambda^\pm_\phi)$ reaches its maximum, the relative free factor support $\F_\supp(\Lambda^\pm_\phi,\F)$ also reaches its maximum which must be $\{[F_n]\}$. It follows that $\phi$ is fully irreducible relative to~$\F$, completing Theorem~I if $\h$ is not geometric above~$\F$.

In the case that $\h$ is geometric above~$\F$, only the first ping-pong tournament is needed. When a lamination pair $\Lambda^\pm_\phi$ is not geometric then the nonattracting subgroup system $\A_\na(\Lambda^\pm_\phi)$ is not a free factor system but it is a vertex group system (see Section~\refGM{SectionVertexGroupSystems}). The descending chain condition (see Proposition~\refGM{PropVDCC}) shows that general vertex group systems are subject to induction: every strictly descending chain of vertex group systems must terminate. At the conclusion of Proposition~\ref{PropUniversallyAttracting} one obtains $\phi \in \h$ and $\Lambda^\pm_\phi$ such that $\A_\na\Lambda^\pm_\phi$ attains its minimal value. Using that $\h$ is geometric above~$\F$ one obtains automatically that $\F_\supp(\Lambda^\pm_\phi,\F)$ automatically attains its maximal value, obviating the need for a second ping-pong tournament. Again one concludes that $\phi$ is fully irreducible rel~$\F$. This ``automatic maximality'' phenomenon generalizes a familiar feature of subgroup classification theory for the mapping class group $\MCG(S)$ of a finite type surface $S$ \cite{Ivanov:subgroups}. Given a subsurface $A \subset S$ with connected complement $B=S-A$, and given $\psi \in \MCG(S)$ leaving both $A$ and $B$ invariant up to isotopy, if $\psi$ is pseudo-Anosov on some subsurface of $B$ with a certain stable-unstable lamination pair $\Lambda^s,\Lambda^u$, and if $A$ is the ``nonattracting subsurface'' for $\Lambda^s,\Lambda^u$ (i.e.\ a simple closed curve is not attracted to $\Lambda^u$ under positive iterates of $\psi$ if and only if that curve is isotopic into $A$) then $\psi$ is automatically pseudo-Anosov on the entirety of~$B$. As a special case, if $\bdy S \ne \emptyset$ and $A$ is a regular neighborhood of $\bdy S$, and so $B$ is isotopic to $S$ itself, then $\psi$ is pseudo-Anosov on the entirety of $S$ if and only if the only curves not attracted to $\Lambda^s,\Lambda^u$ are the components of~$\bdy S$.

In both the geometric and nongeometric cases, the induction processes of the ping-pong tournaments are controlled by stabilizer groups. When driving down the nonattracting subgroup system $\A_\na(\Lambda^\pm_\phi)$ to its minimal value in the first tournament, the induction continues if and only if the subgroup of $\h$ that stabilizes $\A_\na(\Lambda^\pm_\phi)$ has infinite index in~$\h$; when the index is finite and the induction is complete, the minimal value of $\A_\na(\Lambda^\pm_\phi)$  is proved to have the desired form (see Case~1 of the proof of Proposition~\ref{PropUniversallyAttracting}). When driving up the free factor system $\F_\supp(\Lambda^\pm_\phi)$ in the second tournament, the induction continues if and only if the subgroup that stabilizes $\F_\supp(\Lambda^\pm_\phi)$ is proper in~$\h$; when the stabilizer subgroup equals~$\h$ and the induction is complete, the maximal value of the relative free factor support $\F_\supp(\Lambda^\pm_\phi,\F) = \F$ is deduced (see Case~1 of the proof of Proposition~\ref{PropDrivingUp}).

\smallskip\textbf{The role of finite generation: How to find the first player.} It is a poor ping-pong tournament when no-one show up to play. In order to apply the methods outlined above to prove Theorem~I, it is necessary at the start to supply some outer automorphism $\phi \in \h$ and a lamination pair $\Lambda^\pm_\phi \in \L^\pm(\phi)$ that is not carried by~$\F$. For this purpose, Theorems~C and~I contain a hypothesis saying that the subgroup $\h$ be finitely generated. The role of this hypothesis is to enable application of the Relative Kolchin Theorem of \PartThree, the conclusion of which is exactly the existence of $\phi \in \h$ and  $\Lambda^\pm_\phi \in \L^\pm(\phi)$ as needed. This is the \emph{only} place in the proofs of Theorem~C and~I where finite generation is used. Theorems~C and~I each have an alternate hypothesis which may be applied instead of finite generation. In the case of Theorem I, in any situation where one already has in hand $\phi \in \h$ and $\Lambda^\pm_\phi \in \L^\pm(\phi)$ such that $\Lambda^\pm_\phi$ is not supported by~$\F$, no finite generation hypothesis is needed and one may start the ping-pong tournament.

\smallskip\textbf{Section \ref{SectionRelGeomIrr}. Theorem J: Relatively geometric irreducible subgroups.} We prove the general, relative version of Theorem J, the absolute version of which is stated in \Intro. To do this, we go one step further in the analysis of the case where $\h$ is geometric above~$\F$, and in which we produced $\phi \in \h$ and a geometric lamination pair $\Lambda^\pm_\phi \in \L^\pm(\phi)$ for which $\A_\na\Lambda^\pm_\phi$ takes on its minimal value, of the form $\F \union \{[C]\}$. In this case, $[C]$ is represented by the top boundary curve $\bdy_0 S$ of the surface $S$ associated to a geometric model for $\phi$ and $\Lambda^\pm_\phi$. We use the logic of the proof of Theorem I to conclude that the stabilizer in $\h$ of $\A_\na\Lambda^\pm_\phi = \F \union \{[C]\}$ has finite index in $\h$, and therefore must equal~$\h$. We then apply Proposition~\refGM{PropVertToFree} to conclude that the entire subgroup $\h$ preserves the surface $S$ and its boundary components, inducing a homomorphism $\h \to \MCG(S)$ under which the image of $\phi$ is pseudo-Anosov.

\setcounter{tocdepth}{2}
\tableofcontents

\section{Ping-pong on geodesic lines}
\label{SectionPingPong}

\subsection{Finding attracting laminations}
\label{SectionFindingAttrLams}

Given $\phi \in \Out(F_n)$ there are several methods for finding an attracting lamination of $\phi$. One method is to take a relative train track representative $f \from G \to G$ and check the existence of an \eg\ stratum $H_r \subset G$ (Fact~\refGM{FactLamsAndStrata}). A second method, less concrete, is to check existence of a nontrivial conjugacy class $c$ such that for some (any) marked graph $G$, the length in $G$ of the conjugacy class represented by $\phi^i(c)$ is bounded below by an exponentially growing function of the exponent~$i$; the proof indirectly depends on relative train track theory, by noting that if a relative train track representative $f \from G \to G$ of $\phi$ has no \eg\ strata then for any circuit $\gamma$ in $G$ the number of edges in $f^k_\#(\gamma)$ has a polynomial upper bound in~$k$. 

In Lemma~\ref{FindingEG}, using a topological representative $f \from G \to G$ which need not be a relative train track map, we give a third method: from any path which maps over itself three times (in the sense of the $\#\#$~Lemma~\refGM{LemmaDoubleSharpFacts}) one obtains an attracting lamination. The proof uses the definition of attracting laminations directly. 

\smallskip
\textbf{Remarks.} The three methods just described give different amounts of information on the side regarding the attracting lamination that is produced. Relative train track maps give the most information: using filtration elements one obtains certain free factor systems which do and do not support the lamination; using edges of the \eg\ stratum $H_r$ one produces attracting neighborhoods of the lamination; and one can construct the nonattracting subgroup system of the lamination (Definitions~\refWA{defn:Z} and Corollary~\trefWA{CorPMna}{ItemAnaDependence}). The other two methods, including Lemma~\ref{FindingEG}, are useful when no relative train track map is available and when less extra information on the side is needed, although Lemma~\ref{FindingEG} will produce a useful attracting neighborhood.

\medskip

Recall from Section~\refGM{def:DoubleSharp} that any $\pi_1$-injective map $f \from K \to G$ of marked graphs naturally induces two path maps $f_\#, f_{\#\#} \from \wh\B(K) \to \wh\B(G)$ as follows: the path $f_\#(\gamma)$ is obtained by straightening the $f$ image of $\gamma$; and, roughly speaking, $f_{\#\#}(\gamma)$ is the largest common subpath of all $f_\#$-images of paths containing~$\gamma$. Recall also from Section~\refGM{SectionLineDefs} the notation $V(G,\gamma)$ for the basis element of the weak topology on $\B(G)$ associated to a finite path $\gamma$ in a finite graph $G$.

\begin{lemma} \label{FindingEG} 
Given $\phi \in \Out(F_n)$, a marked graph $G$, a topological representative $f \from G \to G$ of $\phi \in \Out(F_n)$, and a finite path $\beta \subset G$, if the path $f_{\#\#}(\beta)$ contains three disjoint copies of $\beta$ then there exists $\Lambda \in \L(\phi)$ and a generic leaf $\lambda$ of $\Lambda$ such that $\phi$ fixes $\Lambda$, $\phi$ fixes $\lambda$ preserving orientation, and $V(G,\beta)$ is an attracting neighborhood for $\Lambda$. Furthermore for any $i \ge 0$ each generic leaf of $\Lambda$ contains $f^i_{\#\#}(\beta)$ as a subpath. 
\end{lemma}

\begin{proof} For any lift $\ti\beta$ of $\beta$ to the universal cover $\wt G$ and for any lift $\ti f \from \wt G \to \wt G$ of $f$, the hypothesis can be restated to say that 
$$\ti f_{\#\#}(\ti\beta) = \ti \alpha_1 \, \ti \beta_L \, \ti \alpha_2 \, \ti \beta_C \, \ti \alpha_3 \, \ti \beta_R \, \ti \alpha_4
$$
where $\ti\beta_L,\ti\beta_C,\ti\beta_R$ are translates of $\ti\beta$.

For inductive reasons we write $\beta_0 = \beta$. Choosing a lift $\ti \beta_0$ of $\beta_0$ to the universal cover~$\wt G$, there exists a lift $\ti f: \wt G \to \wt G$ of $f$ and lifts $\ti\beta_L,\ti\beta_R$ of $\ti\beta_0$ such that  
$$\ti f_{\#\#}(\ti \beta_0) = \ti \alpha_{0,1} \, \ti \beta_{0,L} \, \ti \alpha_{0,2} \, \ti \beta_{0} \, \ti \alpha_{0,3} \, \ti \beta_{0,R} \, \ti \alpha_{0,4}
$$
Define 
$$\ti \beta_1 = \ti \beta_{0,L} \, \ti \alpha_{0,2}  \, \ti \beta_{0} \, \ti \alpha_{0,3}  \, \ti \beta_{0,R} \subset \ti f_{\#\#}(\ti \beta_0). 
$$
Combining the definition of $\ti\beta_1$, the hypothesis, and Lemma~\trefGM{LemmaDoubleSharpFacts}{ItemDisjointCopies}, we may write $\ti f_{\#\#}(\ti\beta_1)$~as
$$\ti f_{\#\#}(\ti \beta_1) = \ti \alpha_{1,1}  \, \ti \beta_{1,L} \, \ti \alpha_{1,2}  \, \ti \beta_{1 \, }\ti \alpha_{1,3}  \, \ti \beta_{1,R} \, \ti \alpha_{1,4}
$$
where $\ti \beta_{1,L} \subset \ti f_{\#\#}(\ti \beta_{0,L})$ and $\ti \beta_{1,R} \subset \ti f_{\#\#}(\ti \beta_{0,R})$ are translates of $\ti \beta_1 \subset \ti f_{\#\#}(\ti \beta_0)$. Assuming by induction that 
$$\ti f_{\#\#}(\ti \beta_i) = \ti \alpha_{i,1}  \,  \ti \beta_{i,L}  \, \ti \alpha_{i,2}   \, \ti \beta_{i}  \, \ti \alpha_{i,3}  \,  \ti \beta_{i,R}  \, \ti \alpha_{i,4}
$$ 
where $\ti \beta_{i,L} \subset \ti f_{\#\#}(\ti \beta_{i-1,L})$ and $\ti \beta_{i,R} \subset \ti f_{\#\#}(\ti \beta_{i-1,R})$ are translates of $\ti \beta_i \subset \ti f_{\#\#}(\ti \beta_{i-1})$, define 
$$\ti \beta_{i+1} = \ti \beta_{i,L}\ti \alpha_{1,2} \ti \beta_{i}\ti \alpha_{i,3} \ti \beta_{i,R} \subset \ti f_{\#\#}(\ti \beta_i)
$$ 
and apply Lemma~\trefGM{LemmaDoubleSharpFacts}{ItemDisjointCopies} to complete the induction step. 

The union of the nested sequence $ \ti \beta_0 \subset \ti \beta_1\subset \ti \beta_2 \subset\cdots$ is a line $\ti\lambda \in \wt B(\wt G)$ which $\ti f_\#$ fixes preserving orientation, and so determines a line $\lambda \in \B$ which $\phi$ fixes preserving orientation. Each ray $\wt R$ in $\ti \lambda$ contains a translate of $\ti \beta_i$ for all sufficiently large $i$ and so contains a translate of $\ti \beta_i$ for all $i$. Thus $\lambda$ is birecurrent. If a line $\ti \gamma$ contains $\ti \beta_0$ as a subpath then $\ti f_\#(\ti \gamma)$ contains $\ti f_{\#\#}(\ti \beta_0)$ by Lemma~\trefGM{LemmaDoubleSharpFacts}{ItemDblSharpContain} and so contains $\ti \beta_1$. The obvious induction argument shows that $\ti f^i_\#(\ti \gamma)$ contains $\ti \beta_i$ for all $i$. This proves that $V(G,\beta) \subset \B(G) \approx \B$ is an attracting neighborhood for $\lambda$ in $\B$ with respect to the action of $\phi$. Since the length of $\ti f_{\#\#}(\ti \beta_i)$ is at least three times the length of $\ti \beta_i$, the line $\lambda$ is not the axis of a covering translation. By the definition of attracting laminations (Definition~3.1.5 of \BookOne, or see Definition~\refGM{DefAttractingLaminations}) the weak closure $\Lambda \subset \B$ of the line~$\lambda$ is an attracting lamination for $\phi$ and $\lambda$ is a generic leaf of $\Lambda$. Since $V(G,\beta)$ is an attracting neighborhood for $\lambda$, it follows that $V(G,\beta)$ is an attracting neighborhood for~$\Lambda$. 
\end{proof}

\subsection{The ping-pong argument}
\label{SectionPingPongArgument}
 
In this section we state and prove Proposition~\ref{PropSmallerComplexity}, a technical statement in which our ping-pong arguments are packaged. 

\begin{definition}\label{DefGeometricAboveF}
Given a free factor system $\F$ and a subgroup $\h \subgroup \Out(F_n)$ that preserves~$\F$, we say that \emph{$\h$ is geometric above~$\F$}, or that \emph{$\h$ is geometric relative to~$\F$}, if for each $\phi \in \h$, each nongeometric lamination pair in $\L^\pm(\phi)$ is supported by~$\F$.
\end{definition}

\begin{proposition}
\label{PropSmallerComplexity}
Consider $\F$ a (possibly empty) free factor system, rotationless $\phi,\psi \in \Out(F_n)$ that preserve $\F$, and lamination pairs $\Lambda^\pm_\phi \in \L^\pm(\phi)$, $\Lambda^\pm_\psi \in \L^\pm(\psi)$ having a generic leaf $\lambda^\pm_\phi,\lambda^\pm_\psi$ that are fixed by $\phi^\pm,\psi^\pm$, respectively, with fixed orientation. Assume also the following hypotheses:
\begin{itemize}
\item[(a)] $\F \sqsubset \A_\na\Lambda^\pm_\phi$ and $\F \sqsubset \A_\na\Lambda^\pm_\psi$;
\item[(b)] Either both pairs $\Lambda^\pm_\phi$ and $\Lambda^\pm_\psi$ are non-geometric, or the subgroup $\<\phi,\psi\> \subgroup \Out(F_n)$ is geometric above~$\F$;
\item[(i)] Generic leaves of $\Lambda^+_\psi$ are weakly attracted to $\Lambda^+_\phi$ under iteration by $\phi$. 
\item[(ii)] Generic leaves of $\Lambda^-_\psi$ are weakly attracted to $\Lambda^-_\phi$ under iteration by $\phi^{-1}$. 
\item[(iii)] Generic leaves of $\Lambda^+_\phi$ are weakly attracted to $\Lambda^+_\psi$ under iteration by $\psi$. 
\item[(iv)] Generic leaves of $\Lambda^-_\phi$ are weakly attracted to $\Lambda^-_\psi$ under iteration by $\psi^{-1}$. 
\end{itemize}
Under these hypotheses, we may conclude that there exist attracting neighborhoods $V^\pm_\phi, V^\pm_\psi$ of generic leaves of $\Lambda^\pm_\phi,\Lambda^\pm_\psi$, respectively, and there exists an integer $M$, such that for any $m,n \ge M$ the outer automorphism $\xi = \psi^m \phi^n$ has a $\xi$-invariant lamination pair $\Lambda^\pm_\xi$ such that $\Lambda^\pm_\xi$ is non-geometric if $\Lambda^\pm_\phi$ and $\Lambda^\pm_\psi$ are non-geometric, and the following hold:

\smallskip\noindent
$(1)$ $\F$ is carried by $\A_\na \Lambda^\pm_\xi$ which is carried by each of $\A_\na \Lambda^\pm_\phi$ and $\A_\na \Lambda^\pm_\psi$.

\smallskip\noindent
$(2^+)$ $\psi^m(V^+_\phi) \subset V^+_\psi$

\smallskip\noindent
$(2^-)$ $\phi^{-n}(V^-_\psi) \subset V^-_\phi$

\smallskip\noindent
$(3^+)$ $\phi^n(V^+_\psi) \subset V^+_\phi$

\smallskip\noindent
$(3^-)$ $\psi^{-m}(V^-_\phi) \subset V^-_\psi$

\smallskip\noindent
$(4^+)$ $V^+_\xi:= V^+_\psi $ is an attracting neighborhood of generic leaves of $\Lambda^+_{\xi}$.

\smallskip\noindent
$(4^-)$ $V^-_\xi := V^-_\phi$ is an attracting neighborhood of generic leaves of $\Lambda^-_\xi$.

\smallskip\noindent
$(5^+)$ For any weak neighborhood $U^+_\psi \subset \B$ of a generic leaf of $\Lambda^+_\psi$ there exists an integer $M(U^+_\psi)$ such that if $m,n \ge M(U^+_\psi)$ then a generic leaf of $\Lambda^+_\xi$ is in $U^+_\psi$.

\smallskip\noindent
$(5^-)$ For any weak neighborhood $U^-_\phi \subset \B$ of a generic leaf of $\Lambda^-_\phi$ there exists an integer $M(U^-_\phi)$ such that if $m,n \ge M(U^-_\phi)$ then a generic leaf of $\Lambda^-_\xi$ is in $U^-_\phi$.
\end{proposition}

For the proof we shall need the following lemma which says, roughly speaking, that for any line which is weakly attracted to some $\Lambda^+ \in \L(\phi)$, the realization of that line in any marked graph contains a finite segment which is uniformly attracted to $\Lambda^+$ in an appropriate sense. The proof uses the ``buffered splitting argument'', \BookOne\ Lemma~4.2.2, to obtain finite subpaths of generic leaves which survive under iteration in a very strong sense. 

\begin{lemma}
\label{LemmaUniformSegmentAttraction}
Consider $\phi \in \Out(F_n)$ and $\Lambda \in \L(\phi)$ so that~$\phi(\Lambda)=\Lambda$, a relative train track representative $f \from G \to G$ with \eg\ stratum $H_r \subset G$ corresponding to $\Lambda$, and a generic leaf $\lambda \in \Lambda$ realized in $G$. Consider also a marked graph $K$, a homotopy equivalence $h \from K \to G$ that preserves marking, and a line $\ell \in \B$ realized in $K$. If $\ell$ is weakly attracted to $\lambda$ then for any finite subpath $\tau$ of $\lambda$ in $G$ that begins and ends with edges of $H_r$ there exists $k \ge 0$ and a finite subpath $\alpha$ of $\ell$ in $K$ such that:
\begin{enumerate}
\item\label{ItemUniformLineAttraction}
For any $i \ge 0$ the path $(f^{k+i} \circ h)_{\#\#}(\alpha)$ contains $f^i_\#(\tau)$ as a subpath.
\item\label{ItemUniformSegmentAttraction}
For any $i \ge 0$ and any path $\beta$ in $K$ containing $\alpha$ as a subpath, the path $(f^{k+i} \circ h)_{\#\#}(\beta)$ contains $f^i_\#(\tau)$ as a subpath.
\end{enumerate}
\end{lemma}

\begin{proof} Item~\pref{ItemUniformSegmentAttraction} follows from item~\pref{ItemUniformLineAttraction} by the \#\# Lemma~\trefGM{LemmaDoubleSharpFacts}{ItemDblSharpContain}. 

Consider a finite subpath of $\lambda^+$ in $G$ of the form $\tau_- \tau \tau_+$ and consider another finite path $\gamma$ in $G$ that contains $\tau_- \tau \tau_+$ as a subpath and so can be written in the form $\gamma = \gamma_- \tau \gamma_+$ where $\tau_-$ is a terminal segment of $\gamma_-$ and $\tau_+$ is an initial segment of $\gamma_+$. By Lemma~4.2.2 of \BookOne\ there exists a constant $C_1$ (depending only on~$f$, independent of $\tau_-\tau\tau_+$ and of~$\gamma$) such that if $\tau_-,\tau_+$ each cross at least $C_1$ edges of $H_r$ then for each $i \ge 0$ the path $f^i_\#(\gamma)$ decomposes as $f^i_\#(\gamma) = f^i_\#(\gamma_-) f^i_\#(\tau) f^i_\#(\gamma_+)$. Since $\lambda^+$ is a generic leaf we may choose $\tau_-,\tau_+$ so that this is so. 

Since $\ell$ is weakly attracted to $\lambda^+$, there exists $k \ge 0$ such that the line $(f^k \circ h)_\#(\ell)$, which is the realization of $\phi^k(\ell)$ in $G$, contains $\tau_- \tau \tau_+$ as a subpath. Let $C_2$ be a bounded cancellation constant for the map $f^k \circ h \from K \to G$. Choose $\alpha$ to be a subpath of $\ell$ such that $(f^k \circ h)_\#(\alpha)$ decomposes as an initial subpath of length at least $C_2$ followed by $\tau_- \tau \tau_+$ followed by a terminal subpath of length at least $C_2$. For any subpath $\alpha'$ of $\ell$ that contains $\alpha$ as a subpath, it follows by the bounded cancellation lemma that $\gamma = (f^k \circ h)_\#(\alpha')$ contains $\tau_- \tau \tau_+$ as a subpath, and so for any $i \ge 0$ the path $(f^{k+i} \circ h)_\#(\alpha') = f^i_\#(\gamma)$ contains $f^i_\#(\tau)$ as a subpath. Since this is true for any such $\alpha'$, it follows by definition of the $\#\#$ operator that $(f^{k+i} \circ h)_{\#\#}(\alpha)$ contains $f^i_\#(\tau)$ as a subpath.
\end{proof}

\begin{proof}[Proof of Proposition \ref{PropSmallerComplexity}] As the proof proceeds, we will impose finitely many lower bounds constraining $M$; in the end we take $M$ to be the largest of these bounds. 

Choose \ct\ representatives
$$g_\phi \from G_\phi \to G_\phi \quad\text{and}\quad g_\psi \from G_\psi \to G_\psi
$$
of $\phi$ and $\psi$, respectively, in which $\F$ is realized by a filtration element. Let $H_\phi \subset G_\phi$, $H_\psi \subset G_\psi$ denote the \eg\ strata corresponding to $\Lambda^+_\phi$, $\Lambda^+_\psi$, respectively. Applying Fact~\refGM{FactAttrNhdBasis} pick an attracting neighborhood basis $\V^+_\psi = \{V(G_\psi,\beta_k)\}$ of $\Lambda^+_\psi$ where $\beta_k$ is a nested sequence of finite subpaths exhausting $\lambda^+_\psi$, and each $\beta_k$ begins and ends with edges of $H_\psi$. Similarly pick $\V^+_\phi = \{V(G_\phi,\gamma_k)\}$ with respect to $\Lambda^+_\phi$ and $\lambda^+_\phi$. Pick homotopy equivalences $h_\psi : G_\phi \to G_\psi$ and $h_\phi : G_\psi \to G_\phi$ that respect the markings. 

In the arguments to follow, we often abbreviate phrases like ``path \#1 has path \#2 as a subpath'' to ``path \#1 contains path \#2''.

Let $\tau$ be an edge of $H_\phi$. Applying Lemma~\ref{LemmaUniformSegmentAttraction} with $f=g_\phi$ and $h=\Id_{G_\phi}$ there exists a finite subpath $\alpha$ of $\lambda^+_\phi$ and an $m_0 \ge 0$ such that for all $i \ge 0$ the path $(g^{m_0+i}_\phi)_{\#\#}(\alpha)$ contains $(g^i_\phi)_\#(\tau)$---in this special case one could choose $\alpha$ so that $m_0 = 0$ but this extra precision makes no difference to the argument. By another application of Lemma~\ref{LemmaUniformSegmentAttraction}, there exists $m_1 \ge 0$ and a finite subpath $\beta$ of $\lambda^+_\psi$ such that for all $i \ge 0$ the path $(g^{m_1+i}_\phi h_\phi)_{\#\#}(\beta)$ in $G_\phi$ contains $(g^i_\phi)_\#(\alpha)$. Setting $i = m_0 + k$, for any $k \ge 0$ we have:
\begin{itemize}
\item[$(*)$] the path $(g^{m_1+m_0+k}_\phi h_\phi)_{\#\#}(\beta)$ in $G_\phi$ contains $(g^{m_0+k}_\phi)_\#(\alpha)$, which contains $(g^{m_0+k}_\phi)_{\#\#}(\alpha)$ (by the \#\#~Lemma~\refGM{LemmaDoubleSharpFacts}), which contains $(g^k_\phi)_\#(\tau)$.
\end{itemize}
Furthermore, Lemma~\ref{LemmaUniformSegmentAttraction} allows us to lengthen $\beta$ arbitrarily, and so by the description above of the attracting neighborhood basis $\V^+_\psi$ we may choose $\beta$ so that $V^+_\psi \equiv V(G_\psi,\beta) \in \V^+_\psi$ is an attracting neighborhood of $\Lambda^+_\psi$. As a consequence the path $(g^l_\psi)_\#(\beta)$ contains $\beta$ for each $\ell \ge 0$.

By a similar argument applying \emph{(iii)} and Lemma~\ref{LemmaUniformSegmentAttraction} with the roles of $\phi$ and $\psi$ reversed, there exists $m_2 \ge 0$ and a subpath $\gamma$ of $\lambda^+_{\phi}$ such that $V^+_\phi \equiv V(G_\phi,\gamma)$ is an attracting neighborhood of $\Lambda^+_\phi$, and such that for all $l \ge 0$ the path ${(g_{\psi}^{m_2+l}h_\psi})_{\#\#}(\gamma)$ in $G_\psi$ contains $(g_\psi^l)_\#(\beta)$. Note in particular that for all $m \ge m_2$ we have verified that the path ${(g_\psi^m h_\psi)}_{\#\#}(\gamma)$ contains $\beta$, and so $\psi^m(V^+_\phi) \subset V^+_\psi$ if $m \ge m_2$. We now constrain $M$ so that $M \ge m_2$ (which verifies $(2^+)$).

Since $(g^k_\phi)_\#(\tau)$ converges weakly to the birecurrent line $\lambda^+_\phi$ we may choose $m_3$ so that for all $j \ge 0$ the path $(g^{m_3+j}_{\phi})_{\#}(\tau)$ contains three disjoint copies of $\gamma$. We further constrain $M$ so that $M \ge m_0 + m_1+m_3$. For any $m,n \ge M$, we have
\begin{align*}
m &\ge m_2 \\
n &= m_1 + m_0 + \underbrace{m_3 + j}_{=k}, \quad\text{with $j \ge 0$}
\end{align*}
It follows from $(*)$ that the path ${(g_{\phi}^{n} \, h_\phi})_{\#\#}(\beta)$ contains $(g^{m_3+j}_\phi)_\#(\tau)$, which in turn contains three disjoint copies of $\gamma$ (which verifies~$(3^+)$). It follows furthermore, applying the \#\#~Lemma~\refGM{LemmaDoubleSharpFacts}, that the path $(g^m_\psi \, h_\psi \, g^n_\phi \, h_\phi)_{\#\#}(\beta)$ contains three disjoint copies of $(g^m_\psi \, h_\psi)_{\#\#}(\gamma)$ which contains three disjoint copies of $\beta$. The homotopy equivalence $f_\xi= g_{\psi}^{m} \, {h_\psi} \, g_{\phi}^{n} \, {h_\phi}: G_{\psi} \to G_{\psi}$ represents $\xi = \psi^m \phi^n$, and by applying Lemma~\ref{FindingEG} it follows that $V^+_\xi := V^+_\psi= V(G_\psi,\beta)$ is an attracting neighborhood of an attracting lamination $\Lambda_\xi^+$ of $\xi$ (which verifies~$(4^+)$). 

To verify $(5^+)$, in the previous paragraph we could have taken $m=M$; it follows that $(g_{\psi}^{M} \, {h_\psi} \, g_{\phi}^{n} \, {h_\phi})_{\#\#}(\beta)$ contains $\beta$. For arbitrary $m=M+l \ge M$ it follows that $(f_\xi)_{\#\#}(\beta)$ contains $(g^l_\psi)_{\#\#}(g_{\psi}^{M} \, {h_\psi} \, g_{\phi}^{n} \, {h_\phi})_{\#\#}(\beta)$ (by the $\#\#$-Lemma~\trefGM{LemmaDoubleSharpFacts}{ItemDblSharpComp}), which contains $(g^l_\psi)_{\#\#}(\beta)$. Lemma~\ref{FindingEG} has the additional conclusion that a generic leaf $\lambda^+_\xi$ of $\Lambda^+_\xi$ contains $(f_\xi)_{\#\#}(\beta)$, and so $\lambda^+_\xi$ also contains $(g^l_\psi)_{\#\#}(\beta)$. Since $V(G_\psi,\beta)$ is an attracting neighborhood for $\Lambda^+_\psi$, for any $U^+_\psi$ as in $(5^+)$ it follows that if $l$ is sufficiently large, say $l \ge L(U^+_\psi)$, then any line containing $(g^l_\psi)_{\#\#}(\beta)$ is in $U^+_\psi$. In particular $\lambda^+_\xi \in U^+_\psi$, and so $(5^+)$ follows with $M(U^+_\psi) = M + L(U^+_\psi)$.

We have verified that $\xi$ and $\Lambda^+_\xi$ satisfy properties~$(2^+), (3^+), (4^+)$ and $(5^+)$. By similar arguments, with the roles of $\phi,\psi$ played by $\psi^\inv$, $\phi^\inv$ respectively, applying \emph{(ii), (iv)} in place of \emph{(i), (iii)}, and after constraining $M$ with further lower bounds as necessary, we obtain attracting neighborhoods $V^-_\phi \subset U^-_\phi$ of $\Lambda^-_\phi$ and $V^-_\psi \subset U^-_\psi$ of $\Lambda^-_\psi$ so that if $m,n \ge M$ then $\xi^\inv = \phi^{-n} \psi^{-m}$ has an attracting lamination $\Lambda^-_\xi$ that satisfies~$(2^-), (3^-), (4^-)$ and $(5^-)$.

By item~\prefWA{ItemUniformOWOTO} of Corollary~\refWA{CorOneWayOrTheOther} (Theorem~H) there exists $m_4$ so that if $\nu$ is a line that is neither an element of $V^-_\phi = V^-_\xi$ nor carried by $\A_\na \Lambda^\pm_\phi$ then $\phi^k_\#(\nu) \in V^+_\phi$ for all $k \ge m_4$. Imposing the further constraint $M \ge m_4$, it follows that $\xi_\#(\nu) = \psi^m_\# \phi^n_\#(\nu) \in \psi^m_\#(V^+_\phi) \subset V^+_\psi = V^+_\xi$ which proves that $\nu$ is weakly attracted to~$\Lambda_\xi^+$. We will use this in the following form: every line that is not contained in $V^-_\xi$ and is not weakly attracted to $\Lambda^+_\xi$ under iteration of $\xi$ is carried by $\A_\na\Lambda^\pm_\phi$. 
 
By a completely symmetric argument we may assume, after constraining $M$ with further lower bounds as necessary, that every line that is not contained in $V^+_\xi$ and is not weakly attracted to $\Lambda^-_\xi$ under iteration of $\xi^\inv$ is carried by $\A_\na\Lambda^\pm_\psi$. 

This completes the description of all lower bounds constraining $M$. We note that these lower bounds are all determined by the choices of \cts\ representing $\phi^{\pm 1}$ and $\psi^{\pm 1}$ and by choices of homotopy equivalences preserving marking amongst the domains of those \cts. We now work on various pieces of the proof, couched as properties of the construction. Note that at this point of the proof we do yet not know whether $\Lambda^+_\xi,\Lambda^-_\xi$ form a dual lamination pair for~$\xi$.

\medskip

Note that every line that is contained in $V^+_\xi$ is weakly attracted to $\Lambda^+_\xi$, and every line contained in $V^-_\xi$ is weakly attracted to $\Lambda^-_\xi$. As a consequence we have shown:
\begin{description}
\item[(A)] Every line that is weakly attracted to neither $\Lambda_\xi^+$ nor $\Lambda_\xi^-$ is disjoint from both $V^+_\xi$ and $V^-_\xi$ and so is carried by both $\A_\na\Lambda^\pm_\phi$ and $\A_\na\Lambda^\pm_\psi$. 
\item[(B)] Restricting to periodic lines, every conjugacy class carried by both $\A_\na \Lambda^+_\xi $ and $\A_\na \Lambda^-_\xi $ is carried by both $\A_\na \Lambda^\pm_\phi$ and $\A_\na \Lambda^\pm_\psi$. 
\end{description}
Since a generic leaf of $\Lambda^+_\xi$ realized in $G_\psi$ contains the path $\beta$ which begins and ends with edges of $H_\psi$, and since the filtration element of $G_\psi$ corresponding to $\F$ is below the stratum~$H_\psi$, it follows that a generic leaf of $\Lambda^+_\xi$ is not carried by $\F$. Since $\F$ is fixed by~$\phi$ and~$\psi$, it is also fixed by $\xi = \psi^m \phi^n$, and so the closure of the $\xi$-orbit of any conjugacy class supported by~$\F$ is also supported by~$\F$ and therefore does not contain a generic leaf of $\Lambda^+_\xi$. It follows that no conjugacy class supported by $\F$ is weakly attracted to $\Lambda^+_\xi$ under iteration by $\xi$, and so $\F \sqsubset \A_\na(\Lambda^+_\xi)$. A similar argument applies to $\A_\na(\Lambda^-_\xi)$, and we have shown:
\begin{description}
\item[(C)] $\F$ is carried by both $\A_\na \Lambda^+_\xi$ and $\A_\na\Lambda^-_\xi$. It follows that neither $\Lambda^+_\xi$ nor $\Lambda^-_\xi$ is carried by~$\F$.
\end{description}

Next we turn to some conditional properties of the construction, which can be thought of as pieces of a case analysis of the proof.
\begin{description}
\item[(D)] If both lamination pairs $\Lambda^\pm_\phi$ and $\Lambda^\pm_\psi$ are nongeometric then $\Lambda^+_\xi$ and $\Lambda^-_\xi$ are nongeometric laminations.
\end{description}
Suppose that $\Lambda^\pm_\phi$ and $\Lambda^\pm_\psi$ are nongeometric. Arguing by contradiction, if $\Lambda^+_\xi$ is geometric then by Proposition~\refGM{PropGeomEquiv} there is a finite set of $\xi$-invariant conjugacy classes whose free factor support carries $\Lambda^+_\xi$.  Every $\xi$-invariant conjugacy class is carried by both $\A_\na \Lambda^+_\xi$ and $\A_\na \Lambda^-_\xi$, and hence by (B) is also carried by $\A_\na\Lambda^\pm_\psi$, which is a free factor system since $\Lambda^\pm_\psi$ is nongeometric. It follows that $\Lambda^+_\xi$ is carried by $\A_\na \Lambda^\pm_\psi$. But the realization of a generic leaf of $\Lambda_\xi^+$ in $G_\psi$ contains $\beta$, and so $\Lambda^+_\xi$ is not carried by $\A_\na\Lambda^\pm_\psi$, a contradiction. If $\Lambda^-_\xi$ is geometric then a similar contradiction holds using $\A_\na\Lambda^\pm_\phi$.

Next we show:
\begin{description}
\item[(E)] If both pairs $\Lambda^\pm_\phi$ and $\Lambda^\pm_\psi$ are both nongeometric, or if both laminations $\Lambda^+_\xi$ and $\Lambda^-_\xi$ are geometric, then $\Lambda^+_\xi$ and $\Lambda^-_\xi$ are a dual lamination pair.
\end{description}
Consider the subset of $\L^\pm(\xi)$ consisting of all lamination pairs for $\xi$ that are not supported by~$\F$; let this subset be indexed as $\{\Lambda^\pm_i\}_{i \in I}$. Assuming $\Lambda^+_\xi$ and $\Lambda^-_\xi$ are not dual we have 
$$\Lambda^+_\xi = \Lambda^+_i \quad\text{and}\quad \Lambda^-_\xi = \Lambda^-_j \quad\text{for some}\quad i \ne j \in I
$$
From this we derive a contradiction. There are two cases, depending on whether $\Lambda^+_i \subset \Lambda^+_j$.   

\smallskip

\textbf{Case 1: $\Lambda^+_i \not\subset \Lambda^+_j$.} Since $\Lambda^+_j$ is weakly closed and $\xi$-invariant, it follows that a generic leaf $\sigma \in\Lambda^+_j$  is not weakly attracted to $\Lambda^+_i=\Lambda^+_\xi$ by forward iteration of~$\xi$. Lemma~\trefWA{LemmaThreeNASets}{ItemGenNA} implies that $\sigma$ is also not weakly attracted to $\Lambda^-_j=\Lambda^-_\xi$ by backward iteration of~$\xi$. It follows by (A) that $\sigma$ is carried by $\A_\na\Lambda^\pm_\phi$, and so the lamination $\Lambda^+_j$ is carried by $\A_\na\Lambda^\pm_\phi$ (by Fact~\trefGM{FactLinesClosed}{ItemLinesClosedOneGp}).

In the subcase that the pair $\Lambda^\pm_\phi$ is nongeometric, $\A_\na\Lambda^\pm_\phi$ is a free factor system, and since $\F_\supp(\Lambda^+_j) = \F_\supp(\Lambda^-_j)$ it follows that $\Lambda^-_j$ is carried by $\A_\na\Lambda^\pm_\phi$. 

In the subcase that the pair $\Lambda^\pm_\phi$ is geometric, $\A_\na\Lambda^\pm_\phi$ is not a free factor system, but it is malnormal (Proposition~\trefWA{PropVerySmallTree}{ItemA_naMalnormal}). By hypothesis of (E) the lamination $\Lambda^+_j$ is geometric.
We have verified the hypotheses of Fact~\refGM{FactSpanArgument}, the conclusion of which says that $\Lambda^-_j$ is carried by $\A_\na\Lambda^\pm_\phi$.

In either case, we have proved that any generic leaf $\tau$ of $\Lambda^-_j = \Lambda^-_\xi$ is carried by $\A_\na\Lambda^\pm_\phi$, contradicting that $\tau$ is contained in $V_\xi^- =V_\phi^-$ which is an attracting neighborhood of $\Lambda^-_\phi$. 

\smallskip

\textbf{Case 2: $\Lambda^+_i \subset \Lambda^+_j$.} By Lemma~\refGM{containmentSymmetry} we have an inclusion $\Lambda^-_i \subset \Lambda^-_j$. Since $\Lambda^-_i, \Lambda^-_j$ are distinct elements of $\L(\xi^\inv)$, this inclusion is proper, and so $\Lambda^-_j \not\subset \Lambda^-_i$. Since $\Lambda^-_i$ is weakly closed and $\xi$-invariant, it follows that a generic leaf $\sigma$ of $\Lambda^-_i$ is not weakly attracted to $\Lambda^-_j=\Lambda^-_\xi$ under iteration of~$\xi^\inv$.  Lemma~\trefWA{LemmaThreeNASets}{ItemGenNA} implies that $\sigma$ is not weakly attracted to $\Lambda^+_i=\Lambda^+_\xi$ under iteration of~$\xi$. It follows by (A) that $\sigma$ is carried by $\A_\na\Lambda^\pm_\psi$, and so $\Lambda^-_i$ is carried by $\A_\na\Lambda^\pm_\psi$. As in Case~1, arguing in two subcases depending on geometricity of $\Lambda^\pm_\psi$, and applying Fact~\refGM{FactSpanArgument} in the geometric case, we conclude that $\Lambda^+_i = \Lambda^+_\xi$ is carried by $\A_\na\Lambda^\pm_\psi$. A generic leaf $\tau$ of $\Lambda_\xi^+$ is therefore carried by $\A_\na\Lambda^\pm_\psi$, contradicting that $\tau$ is contained in $V_\xi^+=V_\psi^+$ which is an attracting neighborhood of $\Lambda^+_\psi$. 

The last piece is to show:
\begin{description}
\item[(F)] If $\Lambda^+_\xi$ and $\Lambda^-_\xi$ do form a dual lamination pair for $\xi$ then all the conclusions of Proposition~\ref{PropSmallerComplexity} hold. 
\end{description}
To prove (F), assuming $\Lambda^+_\xi$ and $\Lambda^-_\xi$ are a dual lamination pair it follows by Corollary~\refWA{CorPMna} that $\A_\na\Lambda^-_\xi = \A_\na\Lambda^+_\xi$, which is denoted $\A_\na\Lambda^\pm_\xi$. Item~(D) clearly shows that the lamination pair $\Lambda^\pm_\xi$ is nongeometric if the pairs $\Lambda^\pm_\phi$ and $\Lambda^\pm_\psi$ are nongeometric. The only part of Proposition~\ref{PropSmallerComplexity} yet to be proved is conclusion~(1). Since the set of lines carried by any subgroup system is the closure of the axes of the conjugacy classes that it carries, it follows from (B) that $\A_\na\Lambda^\pm_\xi$ is carried by both $\A_\na\Lambda^\pm_\phi$ and $\A_\na\Lambda^\pm_\psi$, and (C) says that $\F$ is carried by $\A_\na\Lambda^\pm_\xi$, which proves~(1). 

\begin{remark} 
\label{RemarkABCDE}
Up to this point of the proof we have made no use of hypothesis (b) nor have we used any knowledge of whether the statement ``$\Lambda^+_\xi$, $\Lambda^-_\xi$ form a dual lamination pair'' is true, that statement appearing only as a hypothesis in~(F). In particular, (A), (B), (C), (D), (E) and (F) all hold without assuming hypothesis~(b). See Proposition~\ref{PropSmallerComplexityWeak} below.
\end{remark}

\medskip

All that remains is to apply hypothesis~(b), saying that either both pairs $\Lambda^\pm_\phi$ and $\Lambda^\pm_\psi$ are nongeometric or both laminations $\Lambda^-_\xi$, $\Lambda^+_\xi$ are geometric. We may therefore apply~(E), concluding that $\Lambda^-_\xi,\Lambda^+_\xi$ are a dual lamination pair; and we may therefore apply (F), which tells us that all the conclusions of Proposition~\ref{PropSmallerComplexity} hold, completing the proof.
\end{proof}

\begin{remark} The proof that $\Lambda^\pm_\xi \in \L^\pm(\xi)$ is a dual lamination pair depends on applying Fact~\refGM{FactSpanArgument} to a certain lamination pairs and nonattracting subgroup systems: $\Lambda^\pm_j \in \L^\pm(\xi)$ and $\A_\na(\Lambda^\pm_\phi)$ in Case~1; and $\Lambda^\pm_i \in \L^\pm(\xi)$ and $\A_\na(\Lambda^\pm_\psi)$ in Case~2. Other cases where this application falls apart are ruled out by hypothesis (b) of Proposition~\ref{PropSmallerComplexity}, or by the weaker hypotheses of implication (E). The proof of duality of $\Lambda^+_\xi,\Lambda^-_\xi$ in the case that all relevant lamination pairs are geometric, by applying Fact~\refGM{FactSpanArgument}, ultimately depends on the span construction of \BookOne\ Lemma~7.0.7. We know of no analogue of the span construction for nongeometric laminations, and we do not know if the duality of $\Lambda^+_\xi$ and $\Lambda^-_\xi$ would hold without the hypotheses of implication~(E). 

Nonetheless, as Remark~\ref{RemarkABCDE} shows we obtain many of the conclusions of Proposition~\ref{PropSmallerComplexity} without assuming hypothesis (b), and we collect these conclusions here for use elsewhere:
\end{remark}

\begin{proposition}
\label{PropSmallerComplexityWeak}
Given $\F$, $\phi$, $\psi$, $\Lambda^\pm_\phi$, and $\Lambda^\pm_\psi$ as in Proposition~\ref{PropSmallerComplexity}, and assuming hypotheses (a), (i)--(iv) of that proposition (but \emph{not assuming} hypothesis (b)), the following partial conclusions of the proposition hold:  There exist attracting neighborhoods $V^\pm_\phi, V^\pm_\psi$ of $\Lambda^\pm_\phi,\Lambda^\pm_\psi$, respectively, and there exists an integer $M$, such that for any $m,n \ge M$ the outer automorphism $\xi = \psi^m \phi^n$ has a $\xi$-invariant attracting lamination $\Lambda^+_\xi \in \L(\xi)$ and a $\xi$-invariant repelling lamination $\Lambda^-_\xi \in \L(\xi^\inv)$ such that 
$\Lambda^-_\xi$ and $\Lambda^+_\xi$ are both non-geometric if $\Lambda^\pm_\phi$ and $\Lambda^\pm_\psi$ are non-geometric, and 
the following hold:
\begin{description}
\item[$(1^*)$] $\F$ is carried by both $\A_\na \Lambda^+_\xi$ and $\A_\na\Lambda^-_\xi$, and so neither $\Lambda^+_\xi$ nor $\Lambda^-_\xi$ is carried by~$\F$. Also, each of $\A_\na \Lambda^+_\xi$ and $\A_\na \Lambda^-_\xi$ is carried by both $\A_\na\Lambda^\pm_\phi$ and $\A_\na\Lambda^\pm_\psi$. 
\item[$(2^\pm)$, $(3^\pm)$, $(4^\pm)$, and $(5^\pm)$] hold as stated.
\end{description}
Furthermore,
\begin{description}
\item[$(6)$] \emph{All of} the conclusions of Proposition~\ref{PropSmallerComplexity} hold, including that $\Lambda^+_\xi$ and $\Lambda^-_\xi$ are a dual lamination pair of $\xi$, under any of the following conditions: 
\begin{itemize}
\item $\Lambda^\pm_\phi$ and $\Lambda^\pm_\psi$ are both nongeometric; or 
\item $\Lambda^+_\xi$ and $\Lambda^-_\xi$ are both geometric; or 
\item $\Lambda^-_\xi$ and $\Lambda^+_\xi$ are a dual lamination pair of $\xi$.
\end{itemize}
\end{description}
\end{proposition}

\section{Proof of the Main Theorem C}
\label{SectionWholeShebang}

In Section~\ref{SectionReduction} we shall reduce Theorem~C to its special case Theorem~I. The statements of Theorems C and I are found the Introduction of this series \Intro; stronger versions are found at the beginning of this Part~IV. Section~\ref{SectionConstructingConjugator} contains the construction of conjugators needed for application to ping-pong arguments. Section~\ref{SectionProofUnivAttr} contains the argument used to drive down the nonattracting subgroup system of an attracting lamination pair, by applying ping-pong. Section~\ref{SectionLooking} contains the argument used to drive up the relative free factor support of an attracting lamination pair, again by applying ping-pong, and also puts the pieces together to prove Theorem~I. Section~\ref{SectionRelGeomIrr} contains the general statement and the proof of Theorem~J, the absolute case of which was stated in the Introduction of this series \Intro.

\subsection{Reduction to Theorem I.}
\label{SectionReduction}

For proving Theorem~C, consider a subgroup $\h \subgroup \IA_n(\Z/3)$ and an $\h$-invariant multi-edge extension $\F \sqsubset \F'$ relative to which $\h$ is irreducible. It follows by Lemma~\refRK{LemmaFFSComponent} that each component of $\F$ and of $\F'$ is $\H$-invariant. 

We claim that there exists exactly one component $[F'] \in \F'$ that is not a component of~$\F$. The existence of at least one such component follows because the extension $\F \sqsubset \F'$ is proper. Suppose that there are two such components $[F'_1] \ne [F'_2]$. Letting $\F_1$ be the set of components $[F] \in \F$ such that $[F] \sqsubset [F'_1]$, it follows that the free factor system $(\F - \F_1) \union \{[F'_1]\}$ is $\H$-invariant, and it is nested strictly between $\F$ and $\F'$ because $[F'_1]$ is a component but $[F'_2]$ is not. This contradicts that $\H$ is irreducible relative to $\F \sqsubset \F'$. 

Consider $\wh\A = \F \meet [F']$, which equals the maximal subset of~$\F$ such that $\wh\A \sqsubset [F']$. We may represent $\wh\A = \{[A_1],\ldots,[A_K]\}$ where $A_k \subgroup F'$ for each $k$. Letting $[\cdot]'$ denote conjugacy classes in the group $F'$, it follows that $\A = \{[A_1]',\ldots,[A_K]'\}$ is a free factor system in $F'$ and that $\A \sqsubset \{[F']'\}$ is a multi-edge extension. Let $\wh\H \subgroup \Out(F')$ be the image of $\H$ under the restriction homomorphism $\Stab_{\Out(F_n)}(\F') \to \Out(F')$ (see Fact~\refGM{FactMalnormalRestriction}). By construction $\A$ is $\wh\h$-invariant, and $\wh\h$ is irreducible relative to~$\A$. Also, since $H_1(F';\Z/3)$ is an $\h$-invariant free factor of $H_1(F_n;\Z/3)$ and since $\h$ acts trivially on $H_1(F_n;\Z/3)$ it follows that $\wh\h$ acts trivially on $H_1(F';\Z/3)$. Since $\h$ is either finitely generated or some lamination pair of some element of $\h$ is carried by $\F'$ but not by $\F$, it follows that $\wh\h$ is either finitely generated or some lamination pair of some element of $\wh\h$ is not carried by~$\A$. Theorem~I produces some $\hat\phi \in \wh\H$ that is fully irreducible relative to~$\A$, and any of its pre-images $\phi \in \H$ is fully irreducible relative to $\F \sqsubset \F'$, completing the reduction.

\subsection{Constructing a conjugator}
\label{SectionConstructingConjugator}
Ping pong arguments often use two group elements $\phi,\psi$ which are ``independent'' in some sense which guarantees that words in high powers of $\phi$ and $\psi$ produce other interesting group elements. ``Independence'' means different things in different contexts, depending on the application. When the ambient group is acting on $\hyp^n$, independence might mean that $\phi,\psi$ are loxodromic and their axes have disjoint endpoints. When the ambient group is the mapping class group of a surface, independence might mean that the stable/unstable laminations of the pseudo-Anosov components of $\phi$ and $\psi$ are mutually transverse and fill the surface.

Often one is handed only $\phi$, and $\psi$ is then constructed as a conjugate $\psi = \zeta \phi \zeta^\inv$. In order to guarantee that $\phi$ and $\psi$ are ``independent'', the conjugating element $\zeta$ must somehow move $\phi$ ``away from itself'' or make $\phi$ ``transverse to itself''. Examples of this train of thought can be seen in the proof of the Tits alternative in various settings \BookOne, \cite{McCarthy:Tits}, \cite{Ivanov:subgroups} and in the proofs of subgroup decomposition theorems for mapping class groups \cite{Ivanov:subgroups}.

Here is our conjugator constructor lemma, which starts with an element $\phi \in \IA_n(\Z/3)$ and a lamination pair. Under a certain group theoretic hypothesis, the conclusion states the existence of a conjugator $\zeta$ satisfying several properties \pref{ItemNotIntoANA}--\pref{ItemMinusNotToPlus} which in some sense describe how $\zeta$ ``moves $\phi$ away from itself'' or ``makes $\phi$ transverse to itself''. The proof of this lemma borrows heavily from the proof of Lemma~7.0.3 of \BookOne, which plays a similar role in the ping-pong argument of \BookOne\ Proposition~7.0.2.

\begin{lemma}
\label{LemmaConjugatorConstructor}
Given $\h \subgroup \IA_n(\Z/3)$, $\phi \in \h$, and a lamination pair $\Lambda^\pm_\phi \in \L^\pm_\phi$ with generic leaves $\lambda^\pm_\phi$, respectively, there exists $\zeta \in \h$ such that the following hold:
\begin{enumerate}
\item\label{ItemNotIntoANA}
None of the lines $\zeta(\lambda^+_\phi)$, $\zeta(\lambda^-_\phi)$, $\zeta^\inv(\lambda^+_\phi)$, $\zeta^\inv(\lambda^-_\phi)$ is carried by $\A_\na \Lambda^\pm_\phi$;
\item\label{ItemPlusNotToMinus}
$\zeta(\Lambda^+_\phi) \ne \Lambda^-_\phi$;
\item\label{ItemMinusNotToPlus}
$\zeta(\Lambda^-_\phi) \ne \Lambda^+_\phi$;
\item\label{ItemANANotIntoItself}
If $\Stab_\h(\A_\na\Lambda^\pm_\phi)$ has infinite index in $\h$ then $\zeta(\A_\na\Lambda^\pm_\phi) \ne \A_\na \Lambda^\pm_\phi$
\end{enumerate}
\end{lemma}

\begin{proof} We quickly reduce the proof to two sublemmas which we afterwards prove. The first sublemma establishes~\pref{ItemNotIntoANA}. Its proof will depend upon the proof of Lemma~7.0.3 of \BookOne.  
\begin{description}
\item[First Sublemma:] There exists a finite index subgroup $\h_0 \subgroup \h $ such that for any $\zeta \in \h_0$, neither $\zeta(\lambda^+_\phi)$ nor $\zeta(\lambda^-_\phi)$ is carried by $\A_\na\Lambda^\pm_\phi$.
\end{description}
\noindent
Applying the First Sublemma, and after passing to a further finite index subgroup of $\h$ still called $\h_0$, we may assume that \pref{ItemNotIntoANA} holds for any $\zeta \in \h_0$, and that either that $\h_0 \subgroup \Stab(\Lambda^+_\phi)$ or that $\Stab_{\h_0}(\Lambda^+_\phi) = \h_0 \intersect \Stab(\Lambda^+_\phi)$ has infinite index in $\h_0$. Notice that in proving \pref{ItemPlusNotToMinus} we need only be concerned with the case that $\Stab_{\h_0}(\Lambda^+_\phi)$ has infinite index, because if $\h_0 \subgroup \Stab(\Lambda^+_\phi)$ then for any choice of $\zeta \in \h_0$ item \pref{ItemPlusNotToMinus} obviously holds, because $\Lambda^+_\phi \ne \Lambda^-_\phi$. Similarly, passing to still a further finite index subgroup, we may assume either that $\h_0 \subgroup \Stab(\Lambda^-_\phi)$ or that $\Stab_{\h_0}(\Lambda^-_\phi) = \h_0 \intersect \Stab(\Lambda^-_\phi)$ has infinite index in $\h_0$, and in proving item \pref{ItemMinusNotToPlus} we need only be concerned about the infinite index case. 

The second sublemma is a simple result about group actions on sets:

\begin{description}
\item[Second Sublemma:] 
Consider a group $H$. Suppose that $H$ acts on sets $X_1,\ldots,X_M$ and that $x_m \in X_m$ are points whose stabilizers in $H$ have infinite index. Then there is an infinite sequence $g_1,g_2,\ldots \in H$ satisfying the following: for all $1 \le m \le M$ and for all $k \ne l \ge 1$ we have $g_k(x_m) \ne g_l(x_m)$. Equivalently, for any finite set of infinite index subgroups $\{S_1,\ldots,S_M\}$ of $H$ there is an infinite subset of $H$ any two elements of which lie in distinct left cosets of each of $S_1,\ldots,S_M$.
\end{description}

We prove Lemma~\ref{LemmaConjugatorConstructor} as follows. The group $\h_0$ acts on the set $C(\B) = \{\text{closed subsets of~$\B$}\}$ and the two laminations $\Lambda^+_\phi$, $\Lambda^-_\phi$ are each elements of $C(\B)$. Consider the stabilizer subgroups $\Stab_{\h_0}(\Lambda^+_\phi)$, $\Stab_{\h_0}(\Lambda^-_\phi)$, and $\Stab_{\h_0}(\A_\na(\Lambda^\pm_\phi))$. Applying the Second Sublemma to the whichever of these three subgroups has infinite index, we obtain $\zeta \in \h_0$ satisfying conclusions \pref{ItemPlusNotToMinus}, \pref{ItemMinusNotToPlus} and~\pref{ItemANANotIntoItself}.
\end{proof}

\begin{proof}[Proof of First Sublemma] Passing to a positive power we assume $\phi$ is rotationless. Let $\fG$ be a \ct\ representing $\phi$ with \eg\ stratum $H_r$ corresponding to $\Lambda^+_\phi$ chosen so that $[G_r] = \F_\supp(\Lambda^\pm_\phi)$. Recall the following notions from Definition~\refWA{defn:Z}: the subgraph $Z \subset G$, the path $\hat \rho_r$ which is a trivial path if $H_r$ is nongeometric and is the unique closed height~$r$ Nielsen path $\rho_r$ if $H_r$ is geometric, and the  subset $\<Z,\hat\rho_r\> \subset \wh\B(G)$ consisting of all lines which are concatenations of edges of $Z$ and copies of $\rho_r$. Recall also from Lemma~\trefWA{LemmaZPClosed}{item:ZP=NA} and Corollary~\refWA{CorPMna} that a line is in the set $\<Z,\hat\rho_r\>$ if and only if it is carried by $\A_\na(\Lambda^\pm_\phi)$. It therefore suffices to prove the following statement: 
\begin{itemize}
\item There exists a finite index subgroup $\h_0 \subgroup \h$ such that for any $\theta \in \h_0$ and any generic lines $\lambda^\pm_\phi$ for~$\Lambda^\pm_\phi$, the realizations of $\theta(\lambda^+_\phi)$ and $\theta(\lambda^-_\phi)$ in $G$ are not in the subset~$\<Z,\hat\rho_r\>$. 
\end{itemize}
Lemma~7.0.3 of \BookOne, proved on pages 615--620 of \BookOne, is the special case of this statement under the additional hypothesis that the lamination pair $\Lambda^\pm_\phi$ is topmost in $\L^\pm(\phi)$, that being a requirement for defining the subgraph $Z$ in \BookOne. But the proof given there works exactly as stated in our present setting, with the following minor changes. One uses our general definition of $Z$ given in Definition~\refWA{defn:Z}, rather than the special definition in the ``topmost'' case given in the proof of \BookOne\ Proposition 6.0.4; the only property of $Z$ needed to make the proof of Lemma~7.0.3 work is that $Z \intersect G_r = G_{r-1}$, which holds here as it does in \BookOne. And in the geometric case: one uses our (strong) geometric model for $f$ and $H_r$ given in Definition~\refGM{DefGeomModel}, rather than the weak geometric model which suffices for the topmost case; and one uses our generalized span argument contained in Fact~\refGM{FactSpanArgument} (together with malnormality of $\A_\na(\Lambda^\pm_\phi)$ proved in Proposition~\trefWA{PropVerySmallTree}{ItemA_naMalnormal}), in place of the topmost case of the span argument contained in \BookOne\ Corollary 7.0.8.
\end{proof}

%

\begin{proof}[Proof of Second Sublemma] The proof is by induction on $M$ with the $M=1$ case being an immediate consequence of the assumption that the stabilizer of $x_1$ in $H$ has infinite index. 

For the inductive step, assume that there is an infinite sequence $\hat g_1, \hat g_2, \ldots \in H$ such that $\hat g_k(x_m) \ne \hat g_l(x_m)$ for $1 \le m \le M-1$ and for $k \ne l$. If $\{\hat g_l(x_M)\}$ is an infinite set then, after passing to a subsequence, we may assume that $\hat g_k(x_M) \ne \hat g_l(x_M)$ for $k \ne l$. In this case we define $g_k = \hat g_k$ for all~$k$. 
 
We may therefore assume, after passing to a subsequence, that $\hat g_k(x_M) = \bar x_M$ is independent of $k$. Since the $H$-orbit of $\bar x_M$ equals the $H$-orbit of $x_M$, there is an infinite sequence $\{h_s\}$ in $H$ such that $h_s(\bar x_M) \ne h_t(\bar x_M)$ for $s \ne t$. 

We define by induction an increasing function $\alpha \from \N \to \N$ such that $h_{J} \hat g_{\alpha(J)}(x_m) \ne h_j \hat g_{\alpha(j)}(x_m)$ for $j<J$ and $1 \le m \le M-1$. Assume that $\alpha(1),\ldots,\alpha(J-1)$ are defined. For $1 \le m \le M-1$, the points $\hat g_{\alpha(J-1)+k}(x_m)$ take infinitely many values as $k \ge 1$ varies, and so the points $h_{J} \hat g_{\alpha(J-1)+k}(x_m)$ take infinitely many values. We may therefore pick $k \ge 1$ and set $\alpha(J) = \alpha(J-1)+k$ so that for $1 \le m \le M-1$, the point $h_J \hat g_{\alpha(J)}(x_m)$ is different from each of $h_1 \hat g_{\alpha(1)}(x_m),\ldots,h_{J-1} \hat g_{\alpha(J-1)}(x_m)$. This completes the definition of $\alpha$.

Setting $g_j = h_{j} \hat g_{\alpha(j)}$ completes the proof.
\end{proof}

\subsection{Driving down $\A_\na(\Lambdapmp)$}
\label{SectionProofUnivAttr}

In the setting of Theorem~I, where $\h \subgroup \IA_n(\Z/3)$ is irreducible with respect to an $\h$-invariant free factor system $\F$ such that $\F \sqsubset \{[F_n]\}$ is a multi-edge extension, and where either $\h$ is finitely generated or some lamination pair of some element of $\h$ is not supported by~$\F$, the desired conclusion is the existence of $\phi \in \h$ which is irreducible rel $\F$. From the Weak Attraction Theory developed in \PartThree, this conclusion follows if one can exhibit $\phi \in \h$ and a dual lamination pair $\Lambda^\pm_\phi \in \L^\pm_\phi$ such that that the joint free factor support of $\F$ and $\Lambda^\pm_\phi$ is ``large enough'' --- namely is equal to $\{[F_n]\}$ --- and such that the nonattracting subgroup system $\A_\na(\Lambdapmp)$ is ``small enough'' --- namely is equal to either $\F$ (in the nongeometric case) or to the union of $\F$ and a single rank~$1$ component that together with $\F$ fills $\{[F_n]\}$ (in the geometric case). 

In this section we focus on the problem of minimizing $\A_\na(\Lambdapmp)$, with particular attention on attaining the equation $\A_\na(\Lambdapmp)=\F$. Of course this equation implies that $\Lambdapmp$ is a nongeometric lamination pair and so it would be impossible to attain if it so happened that the subgroup $\h$ is geometric above~$\F$ (Definition~\ref{DefGeometricAboveF}). 

Fortunately, as the next proposition shows, in the nongeometric case one can attain the desired equation $\A_\na(\Lambdapmp)=\F$ for some $\phi$, whereas in the geometric case one can just finish off the conclusion of Theorem I in its entirety; furthermore one obtains even stronger conclusions that will be used in Theorem~J (see Section~\ref{SectionRelGeomIrr}). Recall that the vertex group system $\A_\na(\Lambda^\pm_\phi)$ is a free factor system system if and only if $\Lambda^\pm_\phi$ is nongeometric; in particular, if $\Lambda^\pm_\phi$ is geometric then $\A_\na(\Lambda^\pm_\phi)$ is properly carried by its free factor support $\F_\supp(\A_\na(\Lambda^\pm_\phi))$.


\begin{prop}[Driving down $\A_\na(\Lambda^\pm_\phi)$]\label{PropUniversallyAttracting}
Consider a subgroup $\h \subgroup \IA_n(\Z/3)$ and an $\h$-invariant free factor system $\F$ such that $\h$ is irreducible rel~$\F$. Consider also a rotationless $\phi \in \h$ and a lamination pair $\Lambda^\pm_\phi \in \L^\pm(\phi)$ that is not carried by $\F$ and that satisfies the following:

\smallskip
\emph{
\textbf{Geometricity Alternative:} If $\h$ is not geometric above $\F$ then $\Lambda^\pm_\phi$ is not geometric.
}

\medskip\noindent
Under these hypotheses, for any weak neighborhood $U \subset \B$ of a generic leaf of $\Lambda^-_\phi$ there exists $\psi \in \h$ and a lamination pair $\Lambda^\pm_\psi$ with the following properties:
\begin{enumerate}
\item\label{ItemProduceNongeometric} If $\h$ is not geometric above~$\F$ then $\Lambda^\pm_{\psi}$ is nongeometric and $\A_\na(\Lambda^\pm_\psi) = \F$.
\item\label{ItemProduceIrreducible}
If $\h$ is geometric above~$\F$ then $\psi$ is fully irreducible relative to~$\F$, and there is a maximal infinite cyclic $C \subgroup F_n$ such that 
\begin{enumeratecontinue}
\item \label{ItemPUAAnaFC}
$\A_\na\Lambda^\pm_\psi = \F \union \{[C]\}$. 
\item \label{ItemPUAFnSuppAna}
$\F_\supp(\F,\Lambda^\pm_{\psi}) = \F_\supp(\F,[C]) = \{[F_n]\}$.
\item \label{ItemPUAAnaHInvariant}
$[C]$ and $\A_\na(\Lambda^\pm_{\psi}) = \F \union \{[C]\}$ are invariant under the entire group~$\h$.
\end{enumeratecontinue}
\item  \label{item:choices} The set $U$ contains a generic leaf of $\Lambda^-_{\psi}$.
\end{enumerate}
\end{prop}

\subparagraph{Remark.} Items \pref{ItemProduceNongeometric}, \pref{ItemPUAAnaFC}, \pref{ItemPUAFnSuppAna} and \pref{item:choices} are used in our proof of Theorem~I in Section~\ref{SectionLooking}. Item~\pref{ItemPUAAnaHInvariant} is used in the proof of Theorem~J to follow in Section~\ref{SectionRelGeomIrr}.

\medskip

The proof of Proposition~\ref{PropUniversallyAttracting} will be an induction based on the chain condition for vertex group systems, Proposition~\refGM{PropVDCC}. The following proposition organizes the machinery of the inductive step into a general statement:

\begin{propositionUnivAttrInd}
For any subgroup $\h \subgroup \IA_n(\Z/3)$, any proper, $\h$-invariant free factor system $\F$ such that $\h$ is irreducible rel~$\F$, any rotationless $\phi \in \h$, and any lamination pair $\Lambda^\pm_\phi \in \L^\pm(\phi)$ not carried by~$\F$ such that $\h$, $\F$, and $\Lambda^\pm_\phi$ obey the geometricity alternative, the following hold: 
\begin{enumerate}
\item\label{ItemUAIExists} There exists $\zeta \in \cH$ satisfying the conclusions of Lemma~\ref{LemmaConjugatorConstructor}.
\item\label{ItemIUAForAny}
For any $\zeta \in \cH$ satisfying the conclusions of Lemma~\ref{LemmaConjugatorConstructor}, for any generic leaves $\gamma^\pm_\phi$ of $\Lambda^\pm_\phi$ that are fixed by $\phi$ with fixed orientation, for any weak neighborhood $U^- \subset \B$ of $\gamma^-_\phi$, and for any weak neighborhood $U^+ \subset \B$ of $\zeta(\gamma^+_\phi)$, there exists an integer $M \ge 1$ such that for any $m,n \ge M$, the outer automorphism $\xi = (\zeta \phi \zeta^\inv)^m \phi^n$ has a lamination pair $\Lambda^\pm_\xi \in \L^\pm(\xi)$ satisfying the following:
\begin{enumerate}
\item $\F \sqsubset \A_\na(\Lambda^\pm_\xi) \sqsubset \A_\na(\Lambda^\pm_\phi)$, and if $\A_\na\Lambda^\pm_\phi$ has infinite index in $\cH$ then the second inclusion is proper.
\item $\Lambda^\pm_\xi$ is nongeometric if $\Lambda^\pm_\phi$ is nongeometric, and therefore $\h$, $\F$, and $\Lambda^\pm_\xi$ continue to obey the Geometricity Alternative.
\item $U^-$ contains a generic leaf of $\Lambda^-_\xi$ and $U^+$ contains a generic leaf of $\Lambda^+_\xi$.
\end{enumerate}
\end{enumerate}
\end{propositionUnivAttrInd}

\begin{proof} Item~\pref{ItemUAIExists} holds since Lemma~\ref{LemmaConjugatorConstructor} does apply. Choose $\zeta$, $\gamma^\pm_\phi$, and $U^\pm$ as in item~\pref{ItemIUAForAny}. Define $\psi = \zeta \phi \zeta^\inv$, and so we have a dual lamination pair $\Lambda^\pm_\psi = \zeta(\Lambda^\pm_\phi) \in \L^\pm(\psi)$ with generic lines $\gamma^+_\psi = \zeta(\gamma^+_\phi)$ and $\gamma^-_\psi = \zeta(\gamma^-_\phi)$ that are fixed by $\psi,\psi^\inv$ respectively with fixed orientations. The nonattracting subgroup system of this lamination pair satisfies 
$$\F = \psi(\F) \sqsubset  \A_\na\Lambda^\pm_\psi = \zeta(\A_\na\Lambda^\pm_\phi)
$$
and so by Lemma~\ref{LemmaConjugatorConstructor}~\pref{ItemANANotIntoItself} we have $\A_\na\Lambda^\pm_\psi \ne \A_\na\Lambda^\pm_\phi$ if $\Stab_\h(\A_\na\Lambda^\pm_\phi)$ has infinite index in~$\h$. Note that $\F_\supp(\Lambda^\pm_\psi) = \zeta(\F_\supp(\Lambda^\pm_\phi))$ and so the connected free factor systems $\F_\supp(\Lambda^\pm_\psi)$ and $\F_\supp(\Lambda^\pm_\phi)$ have the same rank. Noting also that $\Lambda^\pm_\psi$ is non-geometric if $\Lambda^\pm_\psi$ is non-geometric, and that $\<\phi, \psi\> \subgroup \h$, hypotheses (a) and (b) of Proposition~\ref{PropSmallerComplexity} are therefore satisfied, hypothesis (b) following from the \emph{Geometricity Alternative}.

It remains to check hypotheses \emph{(i)--(iv)} of Proposition~\ref{PropSmallerComplexity}. Hypothesis~\emph{(i)} requires that $\Lambda^+_\psi$ is weakly attracted to $\Lambda^+_\phi$ under iteration by~$\phi$. Note that $\Lambda^-_\phi \not\subset \Lambda^+_\psi$, for otherwise, since $\F_\supp(\Lambda^-_\phi)$ and $\F_\supp (\Lambda^+_\psi)$ have the same rank, by Lemma~3.1.15 of \BookOne\ it would follow that $\Lambda^-_\phi = \Lambda^+_\psi = \zeta(\Lambda^+_\phi)$, contradicting Lemma~\ref{LemmaConjugatorConstructor}~\pref{ItemPlusNotToMinus}. We conclude that $\Lambda^+_\psi$ does not contain $\gamma^-_\phi$ and hence there exists a weak neighborhood of $\gamma^-_\phi$ which does not contain $\gamma^+_\psi$.
By Lemma~\ref{LemmaConjugatorConstructor}~\pref{ItemNotIntoANA}, the line $\gamma^+_\psi = \zeta(\gamma^+_\phi)$ is not carried by $\A_\na\Lambda^\pm_\phi$. We may therefore apply our weak attraction result, Corollary~\refWA{CorOneWayOrTheOther} (Theorem~H), from which it follows that $\gamma^+_\psi$ is weakly attracted to $\Lambda^+_\phi$ under iteration by~$\phi$. This verifies hypothesis \emph{(i)}. The symmetric arguments show that hypotheses \emph{(ii)}, \emph{(iii)} and \emph{(iv)} are also satisfied. 

We have now verified all of the hypotheses of Proposition~\ref{PropSmallerComplexity}, from the conclusion of which there exists $M \ge 1$ such that for all integers $m,n \ge M$, the outer automorphism $\xi = \psi^m \phi^n = (\zeta \phi \zeta^{-1})^m \phi^n$ has a lamination pair $\Lambda_\xi^\pm$ such that $\Lambda_\xi^\pm$ is non-geometric if $\Lambda_\phi^\pm$ is non-geometric, such that $\F \sqsubset \A_\na\Lambda^\pm_\xi$, such that $\A_\na \Lambda^\pm_\xi \sqsubset \A_\na \Lambda^\pm_\phi$ and $\A_\na \Lambda^\pm_\xi \sqsubset \A_\na \Lambda^\pm_\psi$, and such that $U^-$ contains a generic leaf of~$\Lambda^-_\xi$ and $U^+$ contains a generic leaf of~$\Lambda^+_\xi$. All that remains to verify amongst the conclusions is that if $\Stab_\cH(\A_\na\Lambda^\pm_\phi)$ has infinite index in $\cH$, then $\A_\na(\Lambda^\pm_\xi)$ is properly contained in $\A_\na(\Lambda^\pm_\phi)$. Since $\zeta(\A_\na \Lambda^\pm_\phi) = \A_\na \Lambda^\pm_\psi$ it follows that the maximum length of a strictly decreasing sequence of vertex group systems beginning with $\A_\na\Lambda^\pm_\phi$ is the same as the maximum length of a strictly decreasing sequence of vertex group systems beginning with $\A_\na\Lambda^\pm_\psi$, from which it follows that $\A_\na\Lambda^\pm_\phi$ is not strictly contained in $\A_\na \Lambda^\pm_\psi$, and so it is not contained at all since they are unequal (as seen earlier using that $\Stab_\cH(\A_\na\Lambda^\pm_\phi)$ has infinite index in $\cH$). The containment $\A_\na \Lambda^\pm_\xi \sqsubset \A_\na \Lambda^\pm_\phi$ is therefore proper. This completes the proof of the ``Inductive Step of Proposition~\ref{PropUniversallyAttracting}''.
\end{proof}

\begin{proof}[Proof of Proposition~\ref{PropUniversallyAttracting}.] Given~$\phi$, $\Lambda^\pm_\phi$, and $U$ as in the statement of Proposition~\ref{PropUniversallyAttracting}, we assume the following inductive hypothesis: there is a sequence of rotationless elements $\phi = \phi_0,\phi_1,\ldots,\phi_k \in \cH$ and lamination pairs $\Lambda^\pm_\phi = \Lambda^\pm_0,\Lambda^\pm_1,\ldots,\Lambda^\pm_k$ with $\Lambda^\pm_i \in \L^\pm(\phi_i)$ for $0 \le i \le k$, such that each pair satisfies the geometricity alternative, each $\Lambda^-_i$ has a generic leaf contained in $U$, and there is a sequence of proper containments of vertex group systems
$$\A_\na\Lambda^\pm_0 \sqsupset \A_\na\Lambda^\pm_1 \sqsupset \cdots \sqsupset \A_\na\Lambda^\pm_k
$$
By applying Proposition~\refGM{PropVDCC}, the length $k$ of this sequence has an upper bound. 

Suppose that $\Stab_\h(\A_\na\Lambda^\pm_k)$ has infinite index in $\h$. Since periodic generic leaves are dense in $\Lambda^\pm_k$, after replacing $\phi_k$ by a power we may assume that each of $\Lambda^\pm_k$ has a generic leaf that is fixed by $\phi_k$ with fixed orientation. We may therefore apply the Inductive Step of Proposition~\ref{PropUniversallyAttracting} using $\phi_k$ and $\Lambda^\pm_k$, incrementing the length of the sequence by producing $\phi_{k+1}$ and $\Lambda^\pm_{k+1} \in \L^\pm(\phi_{k+1})$ not carried by~$\F$ and satisfying the geometricity alternative, such that there is proper containment $\A_\na\Lambda^\pm_k \sqsupset \A_\na\Lambda^\pm_{k+1}$, and such that a generic leaf of $\Lambda^-_{k+1}$ is contained in $U^-$. 

Since the length $k$ cannot be increased indefinitely, by induction we may henceforth assume that $\Stab_\h(\A_\na\Lambda^\pm_k)$ has finite index in $\h$. Under this assumption we prove that the conclusions of Proposition~\ref{PropUniversallyAttracting} hold with $\psi=\phi_k$ and $\Lambda^\pm_\psi = \Lambda^\pm_k$. Knowing already that a generic leaf of $\Lambda^-_k$ is contained in $U_-$, we are done by applying the following general statement, which will be useful elsewhere:

\begin{lemma} \label{LemmaInvStab} Suppose that $\Gamma \subgroup \IAThree$ is irreducible relative to a free factor system $\F$, and consider a rotationless $\psi \in \Gamma$ and a lamination pair $\Lambda^\pm_\psi \in \L(\psi)$ that is not carried by~$\F$. If $\Stab_\Gamma(\A_\na\Lambda^\pm_\psi)$ has finite index in $\Gamma$, then:   
\begin{enumerate}
\item \label{ItemInvStabAnaInv}
$\A_\na\Lambda^\pm_\psi$ is $\Gamma$-invariant, that is, $\Gamma = \Stab_\Gamma(\A_\na\Lambda^\pm_\psi)$.
\item \label{ItemInvStabNonG}
If $\Lambda^\pm_\psi$ is nongeometric then $\A_\na\Lambda^\pm_\psi = \F$.
\item \label{ItemInvStabG}
If $\Lambda^\pm_\psi$ is geometric then
\begin{enumerate}
\item\label{ItemInvStIrr}
$\psi$ is fully irreducible rel $\F$.
\item\label{ItemInvStFFS}
$\F_\supp(\Lambda^\pm_\psi)$ is $\Gamma$-invariant, that is, $\Gamma = \Stab_\Gamma(\F_\supp(\Lambda^\pm_\psi))$.
\item\label{ItemInvStMax}
There exists a maximal infinite cyclic subgroup $C \subgroup F_n$ not carried by $\F$ such that
\begin{enumerate}
\item \label{ItemAnaFC}
$\A_\na\Lambda^\pm_\psi = \F \union \{[C]\}$. 
\item \label{ItemFnSuppAna}
$\F_\supp(\F,\Lambda^\pm_{\psi}) = \F_\supp(\F,[C]) = \{[F_n]\}$.
\item \label{ItemAnaHInvariant}
$[C]$ and $\A_\na\Lambda^\pm_{\psi} = \F \union \{[C]\}$ are invariant under the entire group~$\Gamma$.
\end{enumerate}
\end{enumerate} 
\end{enumerate}
\end{lemma}

\begin{proof} There are two cases depending on geometricity of $\Lambda^\pm_\psi$.

\smallskip\textbf{Case A: $\Lambda^\pm_\psi$ is nongeometric,} and so $\F' = \A_\na(\Lambda^\pm_\psi)$ is a free factor system (Proposition~\refWA{PropVerySmallTree}). In the chain of free factor systems $\F \sqsubset \F' \sqsubset \{[F_n]\}$, the second inclusion is proper. If the first inclusion is also proper then, since $\Gamma$ is irreducible rel~$\F$, there exists $\theta \in \Gamma$ that does not fix $\F'$. Since $\theta \in \IA_n(\Z/3)$, it follows by Theorem~\refRK{ThmPeriodicFreeFactor} that no nontrivial power $\theta^k$ fixes $\F'$. This shows that $\<\theta\>$ is an infinite cyclic subgroup of $\Gamma$ having trivial intersection with $\Stab_\Gamma(\F')$, and so $\Stab_\Gamma(\F') = \Stab_\Gamma(\A_\na\Lambda^\pm_\psi)$ has infinite index in~$\Gamma$, contradicting the hypothesis of Lemma~\ref{LemmaInvStab}. Thus the first inclusion $\F \sqsubset \F' = \A_\na\Lambda^\pm_\psi$ is not proper, and so $\A_\na\Lambda^\pm_\psi=\F$ is $\Gamma$-invariant, proving all conclusions in Case~A.

\smallskip\textbf{Case B: $\Lambda^\pm_\psi$ is geometric.} For the proof we choose:

\smallskip
\noindent
$\bullet$ A \ct\ $f \from G \to G$ representing $\psi$, with \eg\ stratum $H_s$ corresponding to $\Lambda^\pm_\psi$, and with core filtration element $G_t$ satisfying $[\pi_1 G_t]=\F$ (Theorem~\refGM{TheoremCTExistence}).

\medskip\noindent
With this choice, we note that the following properties hold:

\medskip
\noindent
$\bullet$ $s>t$ --- because $\Lambda^\pm_\psi$ is not carried by $\F$ \, (Fact~\refGM{FactLamsAndStrata}).

\noindent
$\bullet$ $\F \sqsubset \A_\na(\Lambdapmp)$ --- because $\F \sqsubset [\pi_1 G_{s-1}] \sqsubset \A_\na(\Lambda_k^\pm)$ \, (Definition~\refWA{defn:Z}, Corollary~\trefWA{CorPMna}{ItemAnaDependence}).

\smallskip\noindent
The lamination pair $\Lambda^\pm_\psi$ is geometric, so $\A_\na \Lambda^\pm_\psi$ is a vertex group system but not a free factor system  (Proposition~\refWA{PropVerySmallTree}), and $H_s \subset G$ is a geometric stratum. Choose a geometric model for the \ct\ $f \from G \to G$ and the stratum $H_s$ having associated surface subgroup system $[\pi_1 S]$ and rank~1 subgroup systems $[\bdy_0 S],\ldots,[\bdy_m S]$ associated to the components of $\bdy S$, with the notation chosen so that $\bdy_0 S$ represents the same conjugacy class in $F_n$ as the height~$s$ closed indivisible Nielsen path $\rho_s$ (Definition~\refGM{DefGeomModel}), and so that each of $[\bdy_1 S],\ldots,[\bdy_m S]$ are carried by $[\pi_1 G_{s-1}]$. 

By Proposition~\refGM{PropVertToFree}, the action of $\Stab_\Gamma(\A_\na\Lambda^\pm_\psi)$ on conjugacy classes of elements and subgroups fixes $[\pi_1 S]$ and permutes $[\bdy_0 S], [\bdy_1 S],\ldots,[\bdy_m S]$ amongst themselves. It follows that
$$(*) \qquad\qquad \Stab_\Gamma(\A_\na\Lambda^\pm_\psi) \subgroup \Stab_\Gamma[\pi_1 S] \subgroup \Stab_\Gamma(\F_\supp[\pi_1 S]) = \Stab_\Gamma(\F_\supp(\Lambda^\pm_\psi))
$$
where the last equation follows from Proposition~\trefGM{PropGeomEquiv}{ItemBoundarySupportCarriesLambda}. By Theorem~\refRK{ThmPeriodicConjClass}, periodic conjugacy classes are fixed for each element of $\IA_n(\Z/3)$, and so each of $[\bdy_0 S], [\bdy_1 S],\ldots,[\bdy_m S]$ is fixed by each element of $\Stab_\Gamma(\A_\na\Lambda^\pm_\psi)$. Also by Proposition~\trefGM{PropGeomEquiv}{ItemBoundarySupportCarriesLambda}, we have 
$$\F_\supp\{[\bdy_0 S],[\bdy_1 S],\ldots,[\bdy_m S]\} = \F_\supp\{[\pi_1 S]\} = \F_\supp(\Lambda^\pm_\psi)
$$
Consider the following two free factor systems, both of which are stabilized by $\Stab_\Gamma(\A_\na\Lambda^\pm_\phi)$:
\begin{align*}
\F_1 & = \F_\supp(\F \union \{ [\bdy_1 S],\ldots,[\bdy_m S]\}) \\
\F_2 &= \F_\supp(\F \union \{[\bdy_0 S],[\bdy_1 S],\ldots,[\bdy_m S]\}) \\
        & = \F_\supp(\F \union \{ [\pi_1 S]\}) \\
        &= \F_\supp(\F,\Lambda^\pm_\psi)
\end{align*}
By construction we have three inclusions of free factor systems $\F \sqsubset \F_1 \sqsubset \F_2 \sqsubset \{[F_n]\}$. The middle inclusion is proper because $\F_\supp[\bdy_0 S] \not\sqsubset [\pi_1 G_{s-1}]$ (Lemma~\trefGM{LemmaScaffoldFFS}{ItemTopIsNotLower}) but $\F_1 \sqsubset [\pi_1 G_{s-1}]$. If either the first or third inclusion is proper then, taking $\F'=\F_1$ or $\F'=\F_2$ respectively, and applying the same argument as in Case~A, we conclude that the subgroup $\Stab_\Gamma(\F')$ has infinite index in $\Gamma$. Its subgroup $\Stab_\Gamma(\A_\na\Lambda^\pm_\psi)$ therefore also has infinite index, contradicting the hypothesis of Lemma~\ref{LemmaInvStab}. 

We have verified the two equations
$$ \F_1 = \F = [\pi_1 G_t] \qquad\qquad \F_2 = \{[F_n]\}
$$
from the second of which it follows that $\{[F_n]\} = \F_\supp(\F,\Lambda^\pm_\psi)$ which is~\pref{ItemFnSuppAna}. It also follows that $H_s$ is the top stratum because $\F_\supp(\F,\Lambda^\pm_\psi) \sqsubset [\pi_1 G_s]$. Note also that the core filtration element $G_t$ is the core of the subgraph $G_{s-1}$ and so $\F=[\pi_1 G_t] = [\pi_1 G_{s-1}]$, for if not then we may apply Lemma~\trefGM{LemmaScaffoldFFS}{ItemFFSLowerBdys} with the conclusion that $[\pi_1(G_{s-1})] \not\sqsubset \F_\supp\{[\pi_1 G_t],[\pi_1 S]\} = \{[F_n]\}$, a contradiction. Combining this with Definition~\refWA{defn:Z} and Corollary~\refWA{CorPMna} we have 
$$\A_\na(\Lambda^\pm_\phi) = [\pi_1 G_{s-1}] \union \{[\bdy_0 S]\} = \F \union \{[C]\}
$$
where $C \subgroup F_n \approx \pi_1(G)$ is generated by the path homotopy class of $\rho_s$ relative to its base point. This proves~\pref{ItemAnaFC}. It also proves \pref{ItemInvStIrr}, namely that $\phi$ is fully irreducible rel~$\F$, because by (Reduced) in the definition of a \ct, there is no $\phi_k$-periodic free factor system strictly between $\F = [\pi_1 G_t] = [\pi_1 G_{s-1}]$ and $[\pi_1 G_s] = \{[F_n]\}$.

To verify~\pref{ItemAnaHInvariant}, $\A_\na \Lambda^\pm_\psi$ is already invariant under the subgroup $\Stab_\Gamma(\A_\na \Lambda^\pm_\psi)$, which by assumption has finite index in~$\Gamma$. Consider any $\theta \in \Gamma$. Since $\theta$ has a positive power in the subgroup $\Stab_\Gamma(\A_\na\Lambda^\pm_\psi)$ it follows that $\A_\na\Lambda^\pm_\psi$ is periodic under $\theta$. Since $\A_\na\Lambda^\pm_\psi = \F \union \{[\bdy_0 S]\}$ it follows that each of $\F$ and $[\bdy_0 S]$ is $\theta$-periodic. Since $\theta \in \IA_n(\Z/3)$, we may apply Theorem~\refRK{ThmPeriodicFreeFactor} to $\F$ and Theorem~\refRK{ThmPeriodicConjClass} to $[\bdy_0 S]$ and so each of $\F$ and $[\bdy_0 S]$ is fixed by $\theta$. Since $\theta \in \Gamma$ is arbitrary, $\A_\na\Lambda^\pm_\psi$ is invariant under the whole group~$\Gamma$, proving \pref{ItemInvStabAnaInv} in Case~B.

To complete the proof of Lemma~\ref{LemmaInvStab} it remains to verify~\pref{ItemInvStFFS}, and by combining \pref{ItemInvStabAnaInv} and $(*)$ we have
$$\Gamma = \Stab_\Gamma(\A_\na\Lambda^\pm_\psi) \subgroup \Stab_\Gamma(\F_\supp(\Lambda^\pm_\psi)) \subgroup \Gamma 
$$ 
so each of these inclusions is an equation.
\end{proof}
This completes the proof of Proposition~\ref{PropUniversallyAttracting}.
\end{proof}

\subsection{Driving up $\F_\supp(\F,\Lambda^\pm_\phi)$. The proof of Theorem I.}
\label{SectionLooking}

Let $\h \subgroup \IA_n(\Z/3)$ be finitely generated, and let $\F$ be a proper, $\h$-invariant free factor system so that $\F \sqsubset \{[F_n]\}$ is a multi-edge extension and $\h$ is irreducible rel~$\F$. Proposition~\ref{PropUniversallyAttracting} completes the proof of Theorem~I under the assumption that some element of $\h$ has an attracting lamination not supported by~$\F$ and all such laminations are geometric (nonetheless in Section~\ref{SectionRelGeomIrr} we will put more work into the geometric case to obtain  the additional conclusions of Theorem~J). 

We now formulate Proposition~\ref{PropDrivingUp} which picks up where Proposition~\ref{PropUniversallyAttracting} left off in the case that $\h$ is not geometric above~$\F$. Proposition~\ref{PropUniversallyAttracting}~\pref{ItemProduceNongeometric} gives useful information in this case, allowing us to adopt the extra assumption that the nonattracting subgroup system has achieved its minimum value, namely~$\F$.

Whenever $\theta \in \h$ and $\Lambda^\pm_\theta \in \L^\pm(\theta)$ have been specified so that $\Lambda^\pm_\theta$ is not carried by~$\F$, we define the \emph{absolute} and \emph{relative} free factor supports of $\Lambda^\pm_\theta$ to be
\begin{align*}
\F^\absolute_\theta &= \F_\supp(\Lambda^\pm_\theta) \\
\F^\relative_\theta &= \F_\supp(\F,\Lambda^\pm_\theta) = \F_\supp(\F,\F^\absolute_\theta)
\end{align*}

\begin{proposition}[Driving up $\F^\relative_\phi = \F_\supp(\F,\Lambda^\pm_\phi)$]
\label{PropDrivingUp} 
Consider a subgroup $\h \subgroup \IA_n(\Z/3)$ and a proper, $\h$-invariant free factor system $\F$ such that $\h$ is irreducible rel~$\F$. Consider also a rotationless $\phi \in \h$ and a nongeometric lamination pair $\Lambda^\pm_\phi \in \L^\pm(\phi)$ such that \hbox{$\A_\na(\Lambda_\phi^{\pm}) = \F$}. For any weak neighborhood $U \subset \B$ of a generic leaf of $\Lambda^-_\phi$ there exists $\psi \in \h$ and a nongeometric lamination pair $\Lambda^\pm_\psi \in \L^\pm(\psi)$ such that the following properties hold:
\begin{enumerate}
\item $\A_\na(\Lambda^\pm_{\psi}) = \F$.
\item $\F^\relative_{\psi} = \{[F_n]\}$
\item\label{ItemDrivingUpAbsFFSInvariant}
$\F^\absolute_{\psi}$ is $\h$-invariant
\item\label{ItemDrivingUpInU}
$U$ contains a generic leaf of $\Lambda^-_{\psi}$.
\end{enumerate}
In particular, $\psi$ is fully irreducible relative to~$\F$.
\end{proposition}

Delaying the proof for the moment, we now complete:

\begin{proof}[Proof of Theorem~I] Consider $\h \subgroup \IA_n(\Z/3)$ and a proper, $\h$-invariant free factor system $\F$ so that $\F \sqsubset \{[F_n]\}$ is a multi-edge extension and so that $\h$ is irreducible rel~$\F$. 

If  $\h$ is finitely generated we may apply the Relative Kolchin Theorem~\refRK{relKolchin}, with the conclusion that there exists $\theta \in \h$ and a lamination pair $\Lambda_{\theta}^\pm \in \L^\pm(\theta)$ which is not carried by~$\F$. The same conclusion holds if $\h$ is not finitely generated by the hypothesis of Theorem~I as stated at the beginning of Part~IV. We may choose $\theta$ and $\Lambda^\pm_{\theta}$ so that the \emph{Geometricity Alternative} of Proposition~\ref{PropUniversallyAttracting} holds: either $\Lambda^\pm_{\theta}$ is non-geometric or  $\h$ is  geometric above $\F$. Passing to a power if necessary we may assume that $\theta$ is rotationless. Suppose that $U \subset \B$ is a weak neighborhood of a generic leaf of $\Lambda^-_\theta$.

The hypotheses of Proposition~\ref{PropUniversallyAttracting} have now been verified. From its conclusion there exists $\phi' \in \h$ and a lamination pair $\Lambda^\pm_{\phi'} \in \L^\pm(\phi')$  such that  generic leaves of  $\Lambda^-_{\phi}$ are contained in $U$ and such that one of two cases holds: 

\smallskip\textbf{$\Lambda^\pm_{\phi'}$ is geometric,} and $\phi'$ is fully irreducible relative to~$\F$.

\smallskip\textbf{$\Lambda^\pm_{\phi'}$ is nongeometric,} and $\A_\na(\Lambda^\pm_{\phi'}) = \F$.

\smallskip\noindent
In the latter case, passing to a power we may assume that $\phi'$ is rotationless, and the hypotheses of Proposition~\ref{PropDrivingUp} have been verified. From its conclusion there exists $\phi \in \cH$ which is fully irreducible rel~$\F$ and $\Lambda^-_{\phi} \in \L^-(\phi)$ that is not carried by $\F$ and has generic leaf contained in $U$. In both cases, the proof of Theorem~I is complete.
\end{proof}

The proof of Proposition~\ref{PropDrivingUp}, which is structured similarly to the proof of Proposition~\ref{PropUniversallyAttracting}, is an induction based on the chain condition for free factor systems, and we first organize the machinery of the inductive step into a general statement:

\begin{propositionDrivingUpInd}
Consider a subgroup $\h \subgroup \IAThree$, a proper $\h$-invariant free factor system~$\F$ such that $\h$ is irreducible rel~$\F$, a rotationless $\phi \in \h$, and a nongeometric lamination pair $\Lambda^\pm_\phi \in \L^\pm(\phi)$ such that $\A_\na \Lambda^\pm_\phi = \F$. For any $\zeta \in \h$ such that $\{\zeta(\Lambda^+_\phi),\zeta(\Lambda^-_\phi)\} \intersect \{\Lambda^+_\phi,\Lambda^-_\phi\} = \emptyset$, for any generic leaves $\gamma^\pm_\phi$ of $\Lambda^\pm_\phi$ that are fixed by $\phi$ with fixed orientation, for any weak neighborhood $U^- \subset \B$ of $\gamma^-_\phi$, and for any weak neighborhood $U^+ \subset \B$ of $\zeta(\gamma^+_\phi)$, there exists an integer $M \ge 1$ such that for any integers $m,n \ge M$, the outer automorphism $\xi = (\zeta \phi \zeta^\inv)^m \phi^n$ has a nongeometric lamination pair $\Lambda^\pm_\xi \in \L^\pm(\xi)$ satisfying the following:
\begin{enumerate}
\item $\A_\na\Lambda^\pm_\xi = \F$.
\item $\F^\absolute_\phi \sqsubset \F^\absolute_\xi$, and if $\zeta(\F^\absolute_\phi) \ne \F^\absolute_\phi$ then this inclusion is proper.
\item $U^-$ contains a generic leaf of $\Lambda^-_\xi$ and $U^+$ contains a generic leaf of $\Lambda^+_\xi$.
\end{enumerate}
\end{propositionDrivingUpInd}

Before proving this statement, we first apply it:

\begin{proof}[Proof of Proposition~\ref{PropDrivingUp}, assuming its Inductive Step] Consider $\h$, $\F$, $\phi$, $\Lambda^\pm_\phi$, and $U$ as in the statement of the proposition. We assume the following inductive hypothesis: there is a sequence $\phi = \phi_0,\phi_1,\ldots,\phi_k$ and nongeometric lamination pairs $\Lambda^\pm_\phi = \Lambda^\pm_0,\Lambda^\pm_1,\ldots,\Lambda^\pm_k$ such that for $0 \le i \le k$ we have $\Lambda^\pm_i \in \L^\pm(\phi_i)$, $\A_\na(\Lambda^\pm_i) = \F$, each $\Lambda^-_i$ has a generic leaf contained in $U$, and such that there is a sequence of proper containments of free factor systems systems
$$\F^\absolute_{\phi_0} \sqsubset \F^\absolute_{\phi_1} \sqsubset \cdots \sqsubset \F^\absolute_{\phi_k}
$$
By applying Proposition~\refGM{PropVDCC}, the length $k$ of this sequence has an upper bound. 

Suppose that $\Stab_\h(\F^\absolute_{\phi_k})$ is not $\h$-invariant. Choose $\zeta \in \h$ such that $\zeta(\F^\absolute_{\phi_k}) \ne \F^\absolute_{\phi_k}$. Since $\F^\absolute_{\phi_k} = \F_\supp(\Lambda^\pm_k)$, and since $\zeta(\F^\absolute_{\phi_k}) = \F_\supp(\zeta(\Lambda^\pm_k))$, it follows that $\{\zeta(\Lambda^-_k),\zeta(\Lambda^+_k)\} \intersect \{\Lambda^-_k,\Lambda^+_k\} = \emptyset$. After replacing $\phi_k$ by a power we may assume that each of $\Lambda^\pm_k$ has a generic leaf that is fixed by $\phi_k$ with fixed orientation. We may therefore apply the Inductive Step of Proposition~\ref{PropDrivingUp} using $\phi_k$ and $\Lambda^\pm_k$, incrementing the length of the sequence by producing $\phi_{k+1}$ and a nongeometric lamination pair $\Lambda^\pm_{k+1} \in \L^\pm(\phi_{k+1})$ such that $\A_\na(\Lambda^\pm_{k+1}) = \F$, such that there is proper containment $\F^\absolute_{\phi_k} \sqsubset \F^\absolute_{\phi_{k+1}}$, and such that a generic leaf of $\Lambda^-_{k+1}$ is contained in $U$. 

Since the length $k$ cannot be increased indefinitely, by induction we may henceforth assume that $\F^\absolute_{\phi_k}$ is $\h$-invariant, and using this we prove that $\psi=\phi_k$ and $\Lambda^\pm_\psi = \Lambda^\pm_k$ satisfy the conclusions of Proposition~\ref{PropDrivingUp}. The only conclusion that is not yet proved is that $\F^\relative_{\phi_k} = \{[F_n]\}$. Since $\F$ and $\F^\absolute_{\phi_k}$ are $\h$-invariant, it follows that $\F^\relative_{\phi_k} = \F_\supp(\F,\F^\absolute_{\phi_k})$ is $\h$-invariant. Since $\Lambda^\pm_{\phi_k}$ is not supported by~$\F$, it follows that the inclusion $\F \sqsubset \F^\relative_{\phi_k}$ is proper. Since $\h$ is irreducible rel~$\F$, it follows that $\F^\relative_{\phi_k} = \{[F_n]\}$.
\end{proof}

\begin{proof}[Proof of the Inductive Step of Proposition~\ref{PropDrivingUp}.] Consider $\zeta \in \h$ satisfying the property 
$$\{\zeta(\Lambda^-_\phi),\zeta(\Lambda^+_\phi)\} \intersect \{\Lambda^-_\phi,\Lambda^+_\phi\} = \emptyset
$$
(as noted just above the property $\zeta(\F^\absolute_\phi) \ne \F^\absolute_\phi$ is stronger, but the latter property is not yet assumed and will be brought in only in item \pref{ItemProperExtension} below where it is needed). Consider also $\psi  = \zeta \phi \zeta^{-1}$ and the lamination pair $\Lambda_\psi^\pm = \zeta(\Lambda_\phi^{\pm}) \in \L^\pm(\psi)$ which is nongeometric because $\Lambda^\pm_\phi$ is nongeometric. We have $\A_\na(\Lambda_\psi^{\pm}) = \zeta(\A_\na(\Lambda_\phi^{\pm})) = \zeta(\F) = \F = \A_\na(\Lambda^\pm_\phi)$. Hypotheses (a) and (b) of Proposition~\ref{DefGeometricAboveF} have been verified. Also, knowing that $\{\Lambda^-_\psi,\Lambda^+_\psi)\} \intersect \{\Lambda^-_\phi,\Lambda^+_\phi\} = \emptyset$ and that the free factor systems $\F^\absolute_\phi = \F_\supp(\Lambda^\pm_\phi)$ and $\F^\absolute_\psi = \F_\supp(\Lambda^\pm_\psi)$ consist of single components of equal rank, we may use the exact same argument as in the proof of the Inductive Step of Proposition~\ref{PropUniversallyAttracting} to conclude that there are no inclusions amongst the four laminations $\Lambda^-_\psi,\Lambda^+_\psi$, $\Lambda^-_\phi$, $\Lambda^+_\phi$ and that hypotheses (i)--(iv) of Proposition~\ref{DefGeometricAboveF} hold.

Having verified the hypotheses of Proposition~\ref{PropSmallerComplexity}, we may adopt its conclusions and accompanying notation. Thus there exists an integer $M_0 \ge 1$ such that for all integers $M \ge M_0$ and all integers $m(M),n(M) \ge M$ we obtain the following objects:
\begin{enumerate}
\item\label{ItemXiM}
an outer automorphism of the form $\xi_M = \psi^{m(M)} \phi^{n(M)}$;
\item\label{ItemXiMLams}
lamination pairs $\Lambda^\pm_M = \Lambda^\pm_{\xi_M} \in \L^\pm(\xi_M)$;
\item an attracting neighborhood $V^+_M := V^+_\psi$ of a generic leaf of $\Lambda^+_M$ for the action of $\xi_M$;
\item an attracting neighborhood $V^-_M := V^-_\phi$ of a generic leaf of $\Lambda^-_M$ for the action of $\xi^\inv_M$;
\end{enumerate}
such that the following hold:
\begin{enumeratecontinue}
\item\label{ItemANAisStill_F}
$\A_\na(\Lambda^\pm_M) = \F$ (see Proposition~\ref{PropSmallerComplexity}~(1));
\item \label{ItemInVPM}
$\psi^{m(M)}(V^+_\phi) \subset V^+_\psi = V^+_M$;
\item \label{ItemInVMM}
$\phi^{-n(M)}(V^-_\psi) \subset V^-_\phi = V^-_M$;
\item \label{ItemConvergence}
For any neighborhoods $U^+_\psi,U^-_\phi \subset \B$ of generic leaves of $\Lambda^+_\psi$, $\Lambda^-_\phi$ respectively, if $M$ is sufficiently large then 
generic leaves of $\Lambda^+_M$ are contained in $U^+_\psi$ and generic leaves of $\Lambda^-_M$ are contained in~$U^-_\phi$.
\end{enumeratecontinue}
To complete the induction and thereby complete the proof of Proposition~\ref{PropDrivingUp}, knowing items~\pref{ItemXiM},~\pref{ItemXiMLams} and~\pref{ItemANAisStill_F} it remains to prove:
\begin{enumeratecontinue}
\item\label{ItemProperExtension}
If $M$ is sufficiently large then we have an extension $\F^\absolute_\phi \sqsubset \F^\absolute_{\xi_M}$, and if $\zeta(\F^\absolute_\phi) \ne \F^\absolute_\phi$ then this extension is proper.
\end{enumeratecontinue} 
For this purpose it suffices to show that 
\begin{enumeratecontinue}
\item\label{ItemBothIncluded}
If $M$ is sufficiently large then we have extensions $\F^\absolute_\phi \sqsubset \F^\absolute_{\xi_M}$ and $\F^\absolute_\psi \sqsubset \F^\absolute_{\xi_M}$.  
\end{enumeratecontinue}
To see that this suffices, from \pref{ItemBothIncluded} we deduce~\pref{ItemProperExtension} as follows: assuming that the free factor systems $\F^\absolute_\phi = \F_\supp(\Lambda^\pm_\phi)$ and $\F^\absolute_\psi = \zeta(\F^\absolute_\phi) = \zeta(\F_\supp(\Lambda^\pm_\phi))$ are not equal, and noting that those free factor systems each have a single component and that these components have equal rank, it follows that any free factor system that contains both of them contains them properly.


\smallskip

The verification of item~\pref{ItemBothIncluded} will take up the rest of the proof of Proposition~\ref{PropDrivingUp}.

\smallskip

Consider a sequence $(\gamma^-_M)$ of generic leaves of the laminations~$\Lambda^-_M$,~$M \ge M_0$. Let $\Acc^-$ denote the weak accumulation set of this sequence, meaning the set of weak limits of subsequences; this is a closed subset of $\B$, since $\B$ has a countable basis. Consider similarly a sequence $(\gamma^+_M)$ of generic leaves of the laminations $\Lambda^+_M$, and its weak accumulation set~$\Acc^+$. 

We claim that:
\begin{enumeratecontinue}
\item \label{ClaimAccIntersection}
The closed set $\Acc^- \intersect \Acc^+$ is carried by $\F$. 
\end{enumeratecontinue}
Arguing by contradiction, suppose that there exists $\ell \in \Acc^- \intersect \Acc^+$ which is not carried by $\F$. First we show that the weak closure of $\ell$ contains a generic leaf~$\sigma^-_\phi$ of $\Lambda^-_\phi$. If not then $\ell$ is disjoint from some weak neighborhood of $\sigma^-_\phi$. Applying Corollary~\trefWA{CorOneWayOrTheOther}{ItemUniformOWOTO} (Theorem~H), and using that $\ell$ is not carried by $\F = \A_\na(\Lambda^\pm_\phi)$, there exists $N>0$ such that $\phi^N(\ell) \in V^+_\phi$. Since $\ell \in \Acc^-$ it follows that there exists $M \ge M_0$ such that $n(M) \ge M \ge N$ and such that $\phi^N(\gamma^-_{M})\in V^+_\phi$, and so $\phi^{n(M)}(\gamma^-_M) \in V^+_\phi$. By item~\pref{ItemInVPM} we have $\xi^{\vphantom-}_{M}(\gamma_{M}^{-}) = \psi^{m(M)}\phi^{n(M)} (\gamma_{M}^-) \in V^+_{M}$. This contradicts the fact that $\xi^{\vphantom-}_{M}(\gamma_{M}^{-})$, which is a generic leaf of $\Lambda^-_{M}$, is not contained in any attracting neighborhood for $\Lambda^+_{M}$. Having shown that the weak closure of $\ell$ contains $\sigma^-_\phi$, it follows that $\sigma^-_\phi \in \Acc^- \intersect \Acc^+$, and we also know that $\sigma^-_\phi$ is not carried by $\F = \A_\na(\Lambda^\pm_\phi)$. By a completely symmetric argument, using $\sigma^-_\phi$ instead of~$\ell$, using $\xi_M^{-1}= \phi^{-n(M)}\psi^{-m(M)}$ instead of $\xi_M$, and using \pref{ItemInVMM} instead of \pref{ItemInVPM}, it follows that the weak closure of $\sigma^-_\phi$ contains~$\sigma^+_\psi$, and so $ \Lambda^+_\psi \subset \Lambda^-_\phi$, a contradiction that proves Claim~\pref{ClaimAccIntersection}.

\smallskip

Fix a marked graph $H$ having a core subgraph $H_0$ that realizes the free factor system~$\F$. For each $M \ge M_0$, let $A_M \subgroup F_n$ denote a free factor such that $\F^\absolute_{\xi_M} = \{[A_M]\}$. Let $K_M$ be the Stallings graph determined by $H$ and $A_M$, i.e.\ the core of the covering space of $H$ corresponding to the subgroup $A_M$. The restriction $p_M \from K_M \to H$ of the covering map is an immersion such that $[p_*(\pi_1(K_M))] = [A_M]$. 

Since $A_M$ is a free factor  of $F_n$, there exists an embedding of $K_M$ into a marked graph $G_M$ and an extension of $p_M$ to a homotopy equivalence $q_M \from G_M \to H$ that preserves marking. For every line $\ell \in \B$, letting $\ell_M$, $\ell_H$ be its realizations in $G_M$, $H$ respectively, note the following equivalence:
\begin{itemize}
\item[] $\ell$ is supported by $\F_\supp(\Lambda^\pm_{M}) = \F^\absolute_{\xi_M} = \{[A_M]\}$ $\iff$ $\ell_M$ is contained in $K_M$
\end{itemize}
Furthermore, if these equivalent statements hold then: the restriction of $p_M$ to $\ell_M$ is an immersion whose image is $\ell_H$; and $\ell_H$ lifts uniquely via $p_M$ to a line in $K_M$, that line being~$\ell_M$. By construction these statements hold whenever $\ell$ is a leaf of $\Lambda^+_M$ or $\Lambda^-_M$, in particular when $\ell = \gamma^\pm_M$.

Consider the natural cell structure on the graph $K_M$ in which each vertex has \hbox{valence~$\ge 3$}. Consider also the subdivided cell structure with respect to which $p_M \from K_M \to H$ is a cellular map taking each vertex to a vertex and each edge to an edge; the edges with respect to this subdivision are called \emph{edgelets} of $K_M$, and we label each edgelet by its image in $H$. For each $M \ge M_0$ we shall identify the leaves $\gamma^\pm_M$ with their realizations in $G_M$ (obtained by lifting via $p_M$) and so each of $\gamma^\pm_M$ is contained in the subgraph $K_M$ and is immersed by $p_M$ with image being the realizations in $H$. Since $A_M$ is the free factor support  of both $\Lambda^+_{\xi_M}$ and $\Lambda^-_{\xi_M}$,   it follows that 
\begin{enumeratecontinue}
\item\label{ItemCrossingWHKM}
Each natural edge of $K_M$ is crossed by both $\gamma^+_M$ and $\gamma^-_M$. 
\end{enumeratecontinue}

For each integer $C \ge 0$ let $Y_{M,C} \subset K_M$ be the $C$-neighborhood of the set of natural vertices of $K_M$, with respect to a path metric in which every edgelet of $K_M$ has length~$1$. We claim next that:
\begin{enumeratecontinue}
\item\label{ItemYMBound}
There exists a constant $C$ independent of $M \ge M_0$ such that each edgelet of $K_M$ that is labelled by an edge in $H \setminus H_0$ is contained in $Y_{M,C}$, equivalently each edgelet of $K_M \setminus Y_{M,C}$ is labelled by an edge in $H_0$. 
\end{enumeratecontinue}
To prove this, if no such $C$ exists then there is a subsequence $(M_i)_{i \ge 1}$ diverging to $+\infinity$, and for each $i \ge 1$ there is an edgelet $e_i \subset K_{M_i}$ projecting to an edge of $H \setminus H_0$, such that $e_i$ is the central edgelet of an embedded edgelet path $\eta_i \subset K_{M_i}$ of length $2i+1$ and $\eta_i$ contains no natural vertex of $K_{M_i}$. For some subsequence of $(M_i)$, the projection to $H$ of the edgelet $e_i$ is constant independent of~$i$; for some further subsequence of $(M_i)$, the projection to $H$ of the central length~$3$ subpath of $\eta_i$ is constant, independent of~$i$; for some further subsequence the projection of the central length~$5$ subpath is constant; and so on. By continuing inductively and then diagonalizing, we obtain a subsequence of $(M_i)$ such that for each $i$ the projection to $H$ of the central $2i+1$ subsequence of $\eta_j$ is constant independent of $j \ge i$. It follows by \pref{ItemCrossingWHKM} that $\gamma^+_{M_i}$ and $\gamma^-_{M_i}$ each cross $\eta_i$. The nested union of the projections to $H$ of the paths $\eta_i$ is therefore a line $\ell \in \Acc^+ \intersect \Acc^-$ realized in $H$. But $\ell$ crosses an edge of $H \setminus H_0$, namely the projection of $e_i$, and so $\ell$ is not carried by $[H_0]=\F$, contradicting~\pref{ClaimAccIntersection} and therefore proving~\pref{ItemYMBound}.

As a consequence of the fact that $K_M$ has a uniformly bounded number of natural vertices, it follows that the graph $Y_M=Y_{M,C}$ has a uniformly bounded number of edgelets. Note also that the set of edgelet labels---namely, the edges of $H$---is finite, and that $K_M$ has uniformly bounded rank. We may therefore assume, after passing to a further subsequence, that for all $M,M' \ge M_0$ there is a homeomorphism $h_{M,M'} : (K_M,Y_M)  \to (K_{M'},Y_{M'})$ whose restriction to $Y_M$ maps edgelet to edgelet and preserves labels. In other words, as an unlabelled natural graph $K_M$ is independent of $M$, and furthermore its labelled edgelet subgraph $Y_M$ is independent of $M$. The components of $K_M \setminus Y_M$ are central subpaths of natural edges, and all the edgelet labels along these subpaths are in~$H_0$. After passing to another subsequence and perhaps enlarging $Y_M$ we may assume that the edgelet length of each component of $K_M \setminus Y_M$ goes to infinity with $M$.
 
To complete the proof, letting $\sigma^+_\psi \in \Lambda^+_\psi$, $\sigma^-_\phi \in \Lambda^-_\phi$ be generic leaves, it suffices to show that if $M$ is sufficiently large then the realizations in $H$ of both $\sigma^+_\psi$ and $\sigma^-_\phi$ lift into $Y_M$, for that implies that both $\Lambda^+_\psi$ and $\Lambda^-_\phi$ are carried by $[A_M]$, and so $\F^\absolute_\phi, \F^\absolute_\psi \sqsubset \{[A_M]\} = \F^\absolute_{\xi_M}$ as desired for item~\pref{ItemBothIncluded}. Suppose that $\rho$ is a finite subpath of the realization of $\sigma^+_\psi$ in $H$ such that $\rho$ begins and ends with edges in $H\setminus H_0$. Letting $U^+_\psi = \B(H,\rho)$, it follows by \pref{ItemConvergence} that there exists $M_\rho$ such that $\rho$ is a subpath of the realization of $\gamma^+_M$ in $H$, for all $M \ge M_\rho$. Lifting $\rho$ along with $\gamma^+_M$ into $K_M$ we obtain a lift $\rho' \subset K_M$ with first and last edgelets in $Y_M$ and edgelet length $L$ that is independent of $M$.  For sufficiently large $M$, the edgelet length of each component of $K_M\setminus Y_M$ is greater than $L$ and so $\rho' \subset Y_M$. Since $Y_M$ is independent of $M$, each such path $\rho$, and hence the whole line $\sigma^+_\psi$, lifts into $Y_M$ as desired. The symmetric argument for $\sigma^-_\phi$ completes the proof of the Inductive Step of Proposition~\ref{PropDrivingUp}.
\end{proof}

\subsection{Theorem~J: Relatively geometric irreducible subgroups}
\label{SectionRelGeomIrr}
In the case of Proposition~\ref{PropUniversallyAttracting} where $\h$ is geometric above~$\F$, stronger conclusions follow as explained in Theorem~J, the absolute case of which was stated in the introduction as Theorem~I. Here we state and prove Theorem~J in its full generality, which will be pretty quick after we review from \PartOne\ concepts of geometric models needed for the general statement of the theorem.

First we review the definition of a weak geometric model (Definition~\refGM{DefWeakGeomModel}), which applies to any top \eg\ stratum of any \ct. Consider $\phi \in \Out(F_n)$ represented by a \ct\ $f \from G \to G$ with top~\eg\ stratum~$H_r$, and let $\Lambda \in \L(\phi)$ be the attracting lamination that corresponds to~$H_r$. Recall from Definition~\refGM{DefGeometricLamination}, Proposition~\refGM{PropGeomEquiv}, and Definition~\refGM{DefGeometricStratum} that $\Lambda$ is geometric if and only if $H_r$ is geometric if and only if there exists a weak geometric model of the \ct\ $f$ for the stratum $H_r$, the definition of which is as follows. The static data of a weak geometric model consists of a 2-complex $Y$ formed as the quotient of a compact surface $S$ and the graph $G_{r-1}$, where $S$ has one ``upper'' boundary component $\bdy_0 S$ and remaining ``lower'' boundary components $\bdy_i S$, $i=1,\ldots,m$ ($m \ge 0$), and where the quotient is formed by gluing each lower boundary circle $\bdy_i S$ to $G_{r-1}$ using a homotopically nontrivial closed edge path $\alpha_i \from \bdy_i S \to G_{r-1}$. The static data also includes an embedding $G \inject Y$ extending the embedding of $G_{r-1}$, and a deformation retraction $d \from Y \to G$, such that $G \intersect \bdy_0 S = \{p_r\}$ is a single point, and such that $d \restrict \bdy_0 S$ is an immersion. The dynamic data of a weak geometric model consists of a homotopy equivalence $h \from Y \to Y$ and a homeomorphism $\Psi \from (S,\bdy_0 S) \to (S,\bdy_0 S)$ with pseudo-Anosov mapping class $\psi \in \MCG(S)$ such that the maps $(f \restrict G_r) \composed d, \,\, d \composed h \from Y \to G_r$ are homotopic, and the maps $j \composed \Psi, \,\, h \composed j \from S \to Y$ are homotopic. 

Under the quotient map $j \from S \disjunion G_{r-1} \to Y$, the subset $j\left(G_{r-1} \disjunion \bdy_0 S\right) = Y - j(\interior S)$ is called the \emph{complementary subgraph} $K \subset Y$ (Definition~\refGM{DefComplSubgraph}). Since $G_{r-1}$ has no contractible components (by Fact~\refGM{FactGeometricCharacterization}) it follows that $K$ has no contractible components; combining with Lemma~\refGM{LemmaLImmersed}, the restricted map $d \restrict K \from K \to G$ is $\pi_1$-injective on each component of $K$ and there is a subgroup system $[\pi_1 K] $ having one component of the form $[d_* \pi_1(K_l)]$ for each component $K_l$ of $K$. Letting $[\bdy_0 S]$ denote the component of $[\pi_1 K]$ corresponding to the component $\bdy_0 S$ of $K$, and since $G_{r-1}$---the union of the remaining components---represents $\F$, we have the first equation in the following:
$$[\pi_1 K] = \F \union \{[\bdy_0 S]\} = \A_\na(\Lambda^\pm_\phi) \qquad\qquad (*)
$$
For the second equation see Definition~\refWA{defn:Z}; also, this subgroup system is a vertex group system but not a free factor system (Proposition~\refWA{PropVerySmallTree}).

The restricted map $j \restrict S \from S \to Y$ and its composition with $d \from Y \to G$ are $\pi_1$-injective. Picking base points, we get an induced map $dj_* \from \pi_1 S \to F_n$, a $dj_*$-equivariant continuous maps $dj_\bdy \from \bdy \pi_1 S \to \bdy F_n$, and a continuous map $dj_\B \from \B(\pi_1 S) \to \B(F_n)$. The image subgroup $dj_*(\pi_1 S)$ is its own normalizer (Lemma~\refGM{LemmaLImmersed}), so there is a well-defined restriction homomorphism from the subgroup $\Stab[\pi_1 S] \subgroup \Out(F_n)$ to $\Out(\pi_1 S)$ (Fact~\refGM{FactMalnormalRestriction}); we shall denote this homomorphism $dj \big| \from \Stab[\pi_1 S] \to \Out(\pi_1 S)$. It is elementary to check that the induced map of lines $dj_\B$ is equivariant with respect to the restriction homomorphism $dj \restrict$, meaning that for each $\phi \in \Stab[\pi_1 S]$ and each $\ell \in \B(\pi_1 S)$ we have $\phi(dj_\B(\ell)) = dj_\B((dj \restrict \phi)(\ell))$.

The Dehn-Nielsen-Baer Theorem \cite{FarbMargalit:primer} identifies $\MCG(S)$ with the subgroup of $\Out(\pi_1 S)$ preserving the set of conjugacy classes in $\pi_1 S$ associated to oriented components of $\pi_1 S$ (see Proposition~\refGM{PropVertToFree} and the preceding paragraph). Let $\L(\Out(\pi_1 S))$ denote the set of all attracting laminations of all elements of $\Out(\pi_1 S)$, and let $\L(\MCG(S))$ denote the set of all unstable laminations of all elements of $\MCG(S)$. Each element of $\L(\Out(\pi_1 S))$ or of $\L(\MCG(S))$ is regarded as a closed subset of $\B(\pi_1 S)$ (see Section~\refGM{SectionNTReview}). As subsets of $\B(\pi_1 S)$ we have an inclusion $\L(\MCG(S)) \subset \L(\Out(\pi_1 S))$, which can be seen as follows. For the case of an unstable lamination of a pseudo-Anosov mapping class this is proved in Proposition~\refGM{PropGeomLams}. More generally, given $\phi \in \MCG(S)$ and a connected, $\phi$-invariant subsurface $S' \subset S$ on which $\phi$ restricts to a pseudo-Anosov mapping class with unstable lamination $\Lambda \in \L(\MCG(S'))$, by applying the same Proposition~\refGM{PropGeomLams} we may regard $\Lambda$ as an element of $\L(\Out(\pi_1 S'))$, and they by applying Lemma~\refGM{LemmaPushLam} using the inclusion $\pi_1 S' \hookrightarrow \pi_1 S$ we obtain an element of $\L(\Out(\pi_1 S))$. 

\begin{TheoremJ}[Relative, general version] Given a finitely generated subgroup $\h \subgroup \IA_n(\Z/3)$ and an $\h$-invariant free factor system $\F$, if $\F \sqsubset \{[F_n]\}$ is a multi-edge extension, and if $H$ is geometric irreducible rel~$\F$, then there exists $\phi \in \h$ and $\Lambda \in \L(\phi)$ such that $\phi$ is irreducible rel~$\F$ and $\F_\supp(\F,\Lambda)=\{[F_n]\}$. Furthermore, for any such $\phi$ and $\Lambda$, for any \ct\ $f \from G \to G$ with top stratum $H_r$ corresponding to $\Lambda$, and for any geometric model of $f$ and $H_r$ as notated above, the following hold:
\begin{enumerate}
\item\label{ItemThmJStab}
$\h$ stabilizes the subgroup system $[\pi_1 S]$ and the free factor system $\F_\supp[\pi_1 S] = \F_\supp(\Lambda)$, that is, $\h \subgroup \Stab[\pi_1 S]$ and $\h \subgroup \Stab[\pi_1 S] = \Stab(\F_\supp(\Lambda))$.
\item\label{ItemThmJImage}
The image group $dj \restrict \h \subgroup \Out(\pi_1 S)$ is contained in $\MCG(S)$, and by restriction of range we get a homomorphism denoted $dj^\# \from \h \to \MCG(S)$, 
\item\label{ItemThmJPsAn}
$dj^\#(\phi)$ is a pseudo-Anosov mapping class on $S$. More generally, for all $\psi \in \h$, $dj^\#(\psi)$ is pseudo-Anosov if and only if $\psi$ is fully irreducible rel~$\F$.
\item\label{ItemThmJLams}
The map $dj_\B \from \B(\pi_1 S) \to \B(F_n)=\B$ induces bijections amongst the following sets: 
\begin{itemize}
\item the set $\L(\h;\MCG(S))$ consisting of all elements of $\L(\MCG(S))$ which are  unstable laminations of elements of $dj^\#(\h)$ regarded as a subgroup of $\MCG(S)$; 
\item the set $\L(\h;\Out(\pi_1 S))$ consisting of all elements of $\L(\Out(\pi_1 S))$ which are attracting laminations of elements of $dj^\#(\h)$ regarded as a subgroup of $\Out(\pi_1 S))$;
\item the subset $\L(\h; \,\not\sqsubset\!\!\F)$ of $\L(\Out(F_n)) = \union_{\phi \in \Out(F_n)} \L(\phi)$ consisting of those elements of $\L(\h) = \union_{\phi \in \h} \L(\phi)$ that are not supported by $\F$.
\end{itemize}
\item Combining \pref{ItemThmJPsAn} and~\pref{ItemThmJLams}, the map $dj_\B$ also induces a bijection between the set of all unstable laminations of pseudo-Anosov elements of $dj^\#(\h)$ and all attracting laminations not supported by $\F$ of elements of $\h$ that are fully irreducible rel~$\F$.
\end{enumerate}
\end{TheoremJ}

\textbf{Remark.} In order to match the conclusions of the the general, relative case of Theorem~I with the conclusion of the absolute case that was stated in the Introduction \Intro, a few words are needed. In the absolute case the free factor system~$\F$ is trivial, $G_{r-1} = \emptyset$, $H_r = G$, and $Y=S$ has one boundary component~$\bdy_0 S$. In this case we have isomorphisms $\pi_1(S) = \pi_1(Y) \xrightarrow{d_*} \pi_1(G) \approx F_n$ well-defined up to inner automorphism; we have a well-defined induced isomorphism $\Out(\pi_1 S) \approx \Out(F_n)$; the induced homomorphism $\Stab[\pi_1 S] \mapsto \Out(F_n)$ is just the identity map; and the group $\MCG(S)$ may be regarded as a subgroup of $\Out(F_n)$. From the conclusions of general, relative version of Theorem~J it follows that $\h$ is contained in the $\MCG(S)$ subgroup of $\Out(\pi_1 S)$, which is exactly the conclusion of the absolute case.

\begin{proof} The proof starts off just as does the proof of Theorem~I at the beginning of Section~\ref{SectionLooking}, namely, apply the Relative Kolchin Theorem to obtain $\phi' \in \h$ and $\Lambda^\pm_{\phi'} \in \L^\pm(\phi')$ which is not carried by~$\F$, and then apply Proposition~\ref{PropUniversallyAttracting}~\pref{ItemProduceIrreducible} to obtain $\phi \in \h$ and $\Lambda^\pm_\phi \in \L^\pm(\phi)$ such that $\phi$ is irreducible rel~$\F$, $\F_\supp(\F,\Lambda^\pm_\phi) = \{[F_n]\}$, and the subgroup system $\A_\na \Lambda^\pm_\phi$ is stabilized by~$\h$. Fix such a $\phi$ and $\Lambda^\pm_\phi$. Choose a \ct\ $f \from G \to G$ representing $\phi$ with core filtration element $G_s$ such that $\F = [G_s]$ and with \eg-stratum $H_r$ corresponding to $\Lambda^+_\phi$. It follows that $H_r$ is the top stratum, that $H_r$ is a geometric stratum, and from (Filtration) in the definition of a \ct\ that $\F = [G_{r-1}]$, and so $G_s$ is the core of $G_{r-1}$. 

Choose a geometric model for $f$ and $H_r$ denoted as above, and so we have an equation of subgroup systems
$$\A_\na(\Lambda^\pm_\phi) = [\pi_1 K] = [G_{r-1}] \union [\bdy_0 S]
$$
It follows that $\h \subgroup \Stab[\pi_1 K]$. By Proposition~\refGM{PropVertToFree} it also follows that $\Stab[\pi_1 K] \subgroup \Stab[\pi_1 S]$ and that the image of the induced homomorphism $\Stab[\pi_1 K] \mapsto \Out(\pi_1 S)$ is contained in the $\MCG(S)$ subgroup of $\Out(\pi_1 S)$; and so the image of the composed homomorphism $\h \to \Out(\pi_1 S)$ is contained in $\MCG(S)$. This proves~\pref{ItemThmJStab}. Since $\Stab(\F_\supp[\pi_1 S]) \subgroup \Stab[\pi_1 S]$, and since $\F_\supp[\pi_1 S] = \F_\supp(\Lambda)$ (by Proposition~\trefGM{PropGeomEquiv}{ItemBoundarySupportCarriesLambda}), it also proves~\pref{ItemThmJImage}. 

It follows from the definition of a weak geometric model that the image of $\phi$ under the homomorphism $dj^\# \from \h \to \MCG(S)$ is pseudo-Anosov, proving~the first sentence of \pref{ItemThmJPsAn}; we delay the proof of the second sentence.

We next prove \pref{ItemThmJLams}. From the discussion preceding the statement of Theorem~J we obtain the injection $\L(\h;\MCG(S)) \inject \L(\h;\Out(\pi_1 S))$. By applying Lemma~\refGM{LemmaPushLam} we obtain the function $\L(\h;\Out(\pi_1 S)) \to \L(\h)$, and this map is also an injection because attracting laminations of elements of $\L(\h;\MCG(S))$ are disoint form $\B(\bdy S)$ but, by combining Lemma~\trefGM{LemmaLImmersed}{ItemSeparationOfSAndL} with Fact~\refGM{FactBoundaries}, the map $dj_\B \from \B(\pi_1 S) \to \B$ is one--to--one when restricted to $\B(\pi_1 S) - \B(\bdy S)$. 

It remains to prove~\pref{ItemThmJPsAn}. Given $\psi \in \h$, the mapping class $dj^\#(\psi)$ is either reducible or pseudo-Anosov, and we consider these cases separately.

\textbf{Case 1: $dj^\#(\psi)$ is reducible.} We must prove $\psi$ is not fully irreducible rel~$\F$. After passing to a power of $\psi$ we may assume that each component of the Thurston decomposition of $dj^\#(\psi)$ is fixed by $dj^\#(\psi)$. Since $\bdy_0 S$ is invariant by $dj^\#(\psi)$, using the Thurston decomposition of $dj^\#(\psi)$ it follows that there is a connected essential subsurface $F \subset S$ whose isotopy class is invariant by $dj^\#(\psi)$ such that $F \intersect \bdy S$ is a union of components of $\bdy S$, including $\bdy_0 S$ but not including the entirety of $\bdy S$, and such that no component of the (possibly disconnected) essential subsurface $S \setminus F$ is a boundary parallel annulus. The subsurface $S \setminus F$ is also invariant up to homotopy by $dj^\#(\psi)$, and so $G_{r-1} \union (S \setminus F)$ is invariant up to homotopy by $h$. The subgroup system $[G_{r-1} \union (S \setminus F)]$, consisting of the conjugacy classes of the subgroups of $\pi_1(Y) \approx F_n$ which are the $\pi_1$-images of the components of $G_{r-1} \union (S \setminus F)$, is therefore invariant under $\psi$. By elementary surface topology, in the surface $S$ there is a nested triple of finite subgraphs $(\bdy S - \bdy_0 S) \subset A \subset B$ and a deformation retraction $S \mapsto B$ which restricts to deformation retraction $S \setminus F \mapsto A$ such that $A$ is proper in $B$. It follows that there is a deformation retraction $Y = G_{r-1} \union S \mapsto G_{r-1} \union B$ that restricts to a deformation restriction $G_{r-1} \union (S \setminus F) \mapsto G_{r-1} \union A$. Regarding $G_{r-1} \union B$ as a marked graph with respect to the isomorphism $\pi_1(G_{r-1} \union B) \approx \pi_1(Y) \approx F_n$, it follows that the $\psi$-invariant subgroup system $[G_{r-1} \union (S \setminus F)]$ coincides with the free factor system $[G_{r-1} \union A]$ and therefore $\psi$ is not fully irreducible rel~$\F$.

\textbf{Case 2: $dj^\#(\psi)$ is pseudo-Anosov.} Choose a pseudo-Anosov homeomorphism $\Theta_S \from S \to S$ that represents the mapping class $dj^\#(\psi)$. Since $\psi \in \cH$, there exists a topological representative $f_\psi \from (G,G_{r-1}) \to (G,G_{r-1})$. It follows that the maps $f_\psi \circ (d \circ j)$, $(d \circ j) \circ \Theta_S \from S \to G$ are homotopic. Applying Proposition~\refWA{PropWeakGeomRelFullIrr}, we conclude that $\psi$ is fully irreducible rel~$\F$.
\end{proof}

\textbf{Remarks.} With some further work, one will be able to deduce further conclusions in the context of Theorem~J~\pref{ItemThmJPsAn} beyond the case that $dj^\#(\psi)$ is pseudo-Anosov, by relating properties of the Thurston decomposition of $dj^\#(\psi)$ with properties of~$\psi$. For example, by applying Lemma~\refGM{LemmaPushLam} it will follow that $\psi$ is of polynomial growth relative to $\F$ if and only if the Thurston decomposition of $dj^\#(\psi)$ has no pseudo-Anosov components. Also, the subset of lamination pairs in $\L^\pm(\psi)$ that are not supported by $\F$ will be in natural bijective correspondence with the unstable-stable lamination pairs of the Thurston decomposition of~$dj^\#(\psi)$. We do not pursue these issues any further here.

\section{A filling lemma}

In \cite{HandelMosher:SubgroupOutF_n}, the predecessor of this series of papers, we proved Theorem~A which is the special case of Theorem~I under the additional hypothesis that $\F = \emptyset$ (see \Intro). That proof follows the same structure of two ping-pong tournaments followed above in the proof of Theorem~I. In the absolute case, the second tournament has the goal of driving up the ``absolute'' free factor support $\F_\supp(\Lambda^\pm_\phi)$ to its maximal value of $\{[F_n]\}$. The logic of that proof used a more complicated argument for driving up free factor supports, which is encapsulated in Proposition~8.1 of \cite{HandelMosher:SubgroupOutF_n}. Although in proving Theorem~I we have avoided these complications and produced an argument simpler than that in Theorem~A, we nonetheless find that a relativization of \cite{HandelMosher:SubgroupOutF_n} Proposition~8.1 is useful in other contexts \cite{HandelMosher:FreeSplittingLox}, and so we develop that relativization here.

\begin{proposition} \label{two filling sets} Suppose that $B_1, B_2 \subset \B$ are sets of lines, that $\F_0$ is a free factor system that carries  $\cl(B_1) \cap \cl(B_2)$. Suppose also that the following properties hold:
\begin{enumerate}
\item Neither end of any element of $B_1$ or $B_2$ is carried by $\F_0$.  
\item The free factor support of each of $B_1$ and $B_2$ is $\{[F_n]\}$.
\end{enumerate}
Then there exist weak neighborhoods $U(b) \subset \B$, one for each $b \in B_1 \cup B_2$, such that for each proper free factor system $\F $ there exists $b_\F \in B_1 \cup B_2$ such that $\F$ does not carry a line in~$U(b_\F)$. 
\end{proposition}

 
\subparagraph{Remark.} In the absolute case where $\F_0=\emptyset$ this statement is equivalent to Proposition~8.1 of \cite{HandelMosher:SubgroupOutF_n}.

\begin{proof} Choose a marked graph $H$, all of whose vertices have valence at least three, with a core subgraph $H_0$ realizing $\F_0$. Let $L(\cdot)$ denote the edge length of paths in $H$. Let each $b \in B_1 \union B_2$ be realized by a line in $H$ with a chosen base point. For each integer~$C \ge 1$, let $b_C$ be the subpath of $b$ that contains the base point, starts and ends with edges of $H \setminus H_0$, and has exactly $2C$ edges of $H \setminus H_0$, with $C$ edges of $H \setminus H_0$ before the base point and $C$ after it.  The existence of $b_C$ follows from (1).   Let $\wh U(b,C) \subset \wh\B(H)$ denote the weak neighborhood of $b$ consisting of paths in $H$ that  contain $b_C$ as a subpath. Let $U(b,C) = \wh U(b,C) \intersect \B(H) \subset \B(H) \approx \B$.

Consider triples $(G,S,\rho)$ consisting of a marked graph $G$, a proper connected core subgraph $S \subset G$, and a map $\rho \from G \to H$ with the following properties: $\rho$ preserves marking; $\rho$ takes vertices to vertices; $\rho$ is an immersion on each edge of $G$; and $\rho$ is an immersion on the subgraph $S$. It follows that the restriction of $\rho$ to any path in $S$ is a path in $H$. Such a triple is called a \emph{representative} of a proper free factor $F$ if $[F]=[S]$. We put a metric on each edge of $G$ by pulling back the metric on $H$ under the map $\rho$, and we extend the length notation $L(\cdot)$ by setting $L(E)=L(\rho(E))$ for each edge $E \subset G$. Note that a line $\ell \in \B$ is carried by $[F]$ if and only if its realization $\ell_G$ in $G$ is contained in $S$, in which case the restriction of $\rho$ to $\ell_G$ is an immersion whose image is its realization $\ell_H$ in~$H$.

Every proper free factor $F$ has a representative $(G,S,\rho)$. To see why, it is evident that there exists a triple $(G,S,\rho)$ that satisfies all the required properties except that $\rho$ need not be an immersion on $S$. Factor $\rho  \from G \to H$ as a composition of folds as indicated in the following diagram, with certain partial compositions also indicated:
$$\xymatrix{
G = G_0 \ar[r]^{p_1} \ar@/^2pc/[rrr]^{P_m}
  & G_1 \ar[r]^{p_2} 
  & \cdots \ar[r]^{p_m} 
  & G_m \ar[r]^{p_{m+1}} \ar@/_2pc/[rrr]^{\rho_m}
  & \cdots \ar[r]^{p_{M-1}} 
  & G_{M-1} \ar[r]^{p_M} 
  & G_M=H
}$$ 
Mark each graph in this diagram so that all maps preserve marking. By giving precedence to folds involving two edges of the subgraph $S$, we may assume that there exists $J \ge 1$ such that $S_J := P_J(S)$ is a proper subgraph of $G_J$, the restricted map $P_J \restrict S \from S \to S_J$ is a homotopy equivalence, and the restricted map $\rho_J \restrict S_J \from S_J \to H$ is an immersion. After replacing $(G,S,\rho)$ by $(G_J,S_J,\rho_J)$ we obtain a representative of $F$.

We shall prove the proposition using weak neighborhoods of the form $U(b_\F)=U(b_\F,C)$ for some integer $C \ge 1$ independent of~$\F$. If this fails then there exist integers $C_i \to \infty$ and proper free factor systems $\F^i$ represented by $(G^i,S^i,\rho^i)$ such that for every $b \in B_1 \cup B_2$ and every $i$ there exists a line in $S^i$ whose image under $\rho^i$ has $b_{C_i}$ as a subpath; the existence of such a line is equivalent to the existence of a finite path $\beta^i(b) \subset S^i$ whose image $\rho^i(\beta^i(b))$ equals $b_{C_i}$, because every path in $S^i$ extends to a line in $S^i$. Assuming the existence of such \emph{pullback paths} $\beta^i(b)$, we argue to a contradiction.

By the natural simplicial structure on a marked graph,  we mean the one in which each vertex has valence at least three. 

We prove that $L(S^i) \to \infinity$ and so $L(G^i) \to \infinity$. If $L(S^i) \not\to \infinity$ then, after passing to a subsequence, the pair $(S^i,\rho^i)$ is independent of~$i$ in the sense that for any $i,j$ there is a homeomorphism $S^i \leftrightarrow S^j$ that commutes with the maps $\rho^i,\rho^j$ to $H$; in particular it follows that the free factor system $[S^i] = \F^i$ is constant, independent of~$i$. Fixing $b \in B_1 \union B_2$, for each $i$ there are only finitely many choices for the path $\beta^i(b)$ which is a pullback to~$S^i$ of $b_{C_i}$, and since each pullback of each $b_{C_i}$ restricts to a pullback of $b_{C_{i-1}}$ it follows that one can choose the pullback paths $\beta^1(b) \subset \beta^2(b)\subset \beta^3(b)\subset\cdots$ in $S^i$ to be nested and so that their union is a pullback of $b$ to $S^i$ (we refer to this as the ``pullback--nesting'' argument and we use variants of this argument below). This shows that the realization of each $b \in B_1 \union B_2$ in each $G^i$ is carried by $S^i$ and so $b$ is carried by the constant free factor system $\F^i$, contradicting~(2).   

For each $i$ we adopt folding notation as in the above diagram, obtaining maps denoted
$$\xymatrix{
G^i \ar[r]^{P^i_m} \ar@/_1pc/[rr]_{\rho^i} & G^i_m \ar[r]^{\rho^i_m} & H
}
$$ 
If $i$ is sufficiently large, then (2) implies that for each $l=1,2$ the collection of paths 
$\{\rho^i(\beta^i(b)) = b_{C_i} \suchthat b\in B_l\}
$ 
crosses every edge in $H$. We may therefore choose a minimal index $M(i)$ such that for each $l=1,2$ the collection of paths $\{P^i_{M(i)}(\beta^i(b)) \suchthat b \in B_l\}$ crosses every edge in $G^i_{M(i)}$.  


We prove next that $L(G^i_{M(i)}) \to \infinity$. If $L(G^i_{M(i)}) \not\to \infinity$ then, since $L(G^i) \to \infinity$, after passing to a subsequence it follows that $G^i \ne G^i_{M(i)}$, and so the graph $G^i_{M(i)-1}$ is defined. Furthermore since the map $G^i_{M(i-1)} \mapsto G^i_{M(i)}$ is just a fold it follows that $L(G^i_{M(i)-1}) \not\to\infinity$. After passing to a further subsequence, we may assume that the pair $(G^i_{M(i)-1},\rho^i_{M(i)-1})$ is independent of $i$ in the same sense as earlier, and that there is a natural edge $E^i \subset G^i_{M(i)-1}$ independent of $i$ such that $E^i$ is not crossed by the collection of paths $\{P^i_{M(i)-1}(\beta^i(b)) \suchthat b \in B_l\}$ for at least one of $l =1,2$. Fix such an $l$. For each $b \in B_l$, as $i$ increases we may remove uniformly bounded initial and terminal segments of $P^i_{M(i)-1}(\beta^i(b))$ to obtain a path contained in the core of $G^i_{M(i)-1} \setminus E^i$, and by a pullback--nesting argument as above it follows that the realization of $b$ in $G^i_{M(i)-1}$ is contained in the core of $G^i_{M(i)-1} \setminus E^i$. 
But then the free factor system determined by the core of $G^i_{M(i)-1} \setminus E^i$ carries $B_l$, in contradiction to (2). We may therefore assume that $L(G^i_{M(i)}) \to \infty$.  

Each natural edge $E^i$ of $G^i_{M(i)}$ is crossed by some path of the form $P^i_{M(i)}(\beta^i(b))$ for some element $b \in B_1$ and also for some element $b \in B_2$. Since $\rho^i = P^i_{M(i)} \composed \rho^i_{M(i)}$ is an immersion on $S^i$, it is an immersion on the path $\beta^i(b)$, and so $\rho^i_{M(i)}$ is an immersion on $P^i_{M(i)}(\beta^i(b))$. It follows that the path $\rho^i_{M(i)}(E^i)$ in $H$ is a subpath of both an element of $B_1$ and an element of $B_2$. Since $\F_0$ carries  $\cl(B_1) \cap \cl(B_2)$, it follows that there is a constant $K$ so that for each~$i$ and each edge $E^i$ of $G^i_{M(i)}$, either the path $\rho^i_{M(i)}(E^i)$ has length at most $2K$ or $\rho^i_{M(i)}(E^i)$ decomposes as a concatenation of initial and terminal subpaths of length $K$ and a central subpath that is contained in $H_0$.  

After passing to another subsequence, we may assume that for all $i$ there is a subgraph $Y^i \subset G^i_{M(i)}$ such that each component of $G^i_{M(i)}\setminus Y^i$ is an arc (the central subpath of some natural edge) whose length tends to $\infty$ with $i$ and whose image under $\rho^i_{M(i)}$ is contained in $H_0$, and such that the image of each edge of $Y^i$ under $\rho^i_{M(i)}$ has uniformly bounded length in $H$. After passing to another sequence we may assume that the pair $(Y^i,\rho^i_{M(i)})$ is independent of $i$ in the same sense as earlier, and so the free factor system determined by $Y^i$ is independent of~$i$. Fix $k$ and consider any $i > k$. The path $P^i_{M(i)}(\beta^i(b))$, whose $\rho^i_{M(i)}$-image is $b_{C_i}$, has a subpath which is a pullback of $b_{C_k}$, and since $b_{C_k}$ begins and ends with edges not in $H_0$ it follows that this pullback subpath must be contained in $Y^i$, for all sufficiently large $i$. Since $k$ is arbitrary and $(Y^i,\rho^i_{M(i)})$ is independent of $i$, it follows by another pullback--nesting argument that for sufficiently large $i$ the realization of $b$ in $G^i_{M(i)}$ is contained in $Y^i$ and so $b$ is carried by the free factor system determined by $Y^i$, but this free factor system is independent of~$i$, contradicting~(2).
\end{proof}

\bibliographystyle{amsalpha} 
\bibliography{mosher} 

 \end{document}